\documentclass[11pt]{article}
\usepackage[T1]{fontenc}
\usepackage[utf8]{inputenc}
\usepackage[french]{babel}
\usepackage{lmodern}
\usepackage{array}
\usepackage{amsmath}
\usepackage{amsthm}
\usepackage{amssymb}
\usepackage{amsfonts}
\usepackage{fancyhdr}
\usepackage{enumitem}
\usepackage{stmaryrd}
\usepackage{dsfont}
\usepackage{xcolor}
\usepackage{csquotes}
\usepackage{float}
\usepackage{tikz}
\usepackage{tikz-cd}
\usepackage[margin=2.3cm,nohead]{geometry}
\usepackage{hyperref}
\usepackage{cleveref}
\usepackage[affil-it]{authblk}
\usepackage{orcidlink}

\newcommand{\Q}{\mathbb{Q}}
\newcommand{\R}{\mathbb{R}}
\newcommand{\Z}{\mathbb{Z}}
\newcommand{\N}{\mathbb{N}}
\newcommand{\C}{\mathbb{C}}
\newcommand{\h}{\mathcal{H}}
\newcommand{\p}{\mathbb{P}}

\newcommand{\ana}{\mathbf{A}}
\newcommand{\Rat}{\mathrm{Rat}}
\newcommand{\Res}{\mathrm{Res}}
\newcommand{\SL}{\mathrm{SL}}
\newcommand{\res}{\mathrm{res}}
\newcommand{\M}{\mathrm{M}}
\newcommand{\pr}{\mathrm{pr}}
\newcommand{\Frac}{\mathrm{Frac}}
\newcommand{\val}{\mathrm{val}}
\newcommand{\Spec}{\mathrm{Spec}}
\newcommand{\Pre}{\mathrm{Pre}}

\newtheorem{thm}{Théorème}[section]
\newtheorem{prop}[thm]{Proposition}
\newtheorem{nota}[thm]{Notation}
\newtheorem{lem}[thm]{Lemme}
\newtheorem{defi}[thm]{Définition}
\newtheorem{rem}[thm]{Remarque}
\newtheorem{cor}[thm]{Corollaire}
\newtheorem{ex}[thm]{Exemple}

\title{Action de groupe algébrique sur la compactification hybride}
\author{Alexandre Roy \orcidlink{0009-0004-0193-008X}%
  \thanks{Electronic address: \texttt{alexandre.roy@unicaen.fr}}}
\affil{Université Caen Normandie, CNRS, Normandie Univ, LMNO UMR6139, F-14000 CAEN, FRANCE }

\begin{document}

\maketitle 
\begin{abstract}
Soit $X$ une variété affine sur $\C$ et $G$ un groupe algébrique agissant sur $X$ dont l'action est fermée. J. Poineau a défini une compactification $X^\urcorner$ de $X(\C)$ en utilisant les espaces de Berkovich hybrides. On s'intéresse au prolongement de l'action de $G$ sur cette compactification en caractérisant le lieu $\mathcal{U} \subset X^\urcorner$ où l'action est bien définie. On montre également que le quotient de $\mathcal{U}$ par l'action de $G$ est homéomorphe à $(X/G)^\urcorner$, la compactification de $(X/G)(\C)$. On applique ces résultats à $X = \Rat_d$, l'ensemble des fractions rationnelles et $G = \SL_2$. On retrouve alors des résultats de C. Favre-C. Gong dans un contexte plus général. De plus, on obtient une compactification de $\M_d = \Rat_d/\SL_2$ où le bord est constitué d'orbites de fractions rationnelles non-archimédiennes. Ces résultats restent vrais si l'on remplace $\C$ par $k$ un corps non-trivialement valué et l'on remplace les espaces analytiques complexes par des espaces de Berkovich sur $k$ ou si l'on prend $X$ comme étant le lieu stable d'une $k$-variété défini au sens de GIT.
\end{abstract}

\tableofcontents

\section{Introduction}
Soit $X$ une variété sur $\C$ et $G$ un groupe algébrique réductif agissant sur $X$ via $\Phi : X \times G \rightarrow X$. D. Mumford a étudié cette action et le quotient schématique $X/\! /G$ sous réserve d'existence dans \cite{GIT}. Ainsi, si $X$ est une variété affine et $G$ un groupe réductif, alors le schéma $X/\! /G$ existe et est une variété affine. De plus, le morphisme $\pi : X \rightarrow X/\! /G$ est surjectif et $G-$invariant. Si $X$ est une variété non-nécessairement affine, le quotient schématique n'est pas défini en toute généralité. Il l'est néanmoins sur un ouvert de $X$ : le lieu stable défini à partir d'un faisceau inversible sur $X$.

L'objectif de cet article est d'étudier l'action de $G$ sur une compactification de $X(\C)$. Dans cet article, on s'intéresse au quotient de ce lieu stable, alors $X/\! /G$ est un quotient géométrique et sera noté $X/G$ (voir \cite{GIT}). Dans le cas où $X$ est affine et l'action est fermée, alors le quotient géométrique existe pour $X$ entier.  

J. Poineau a construit une compactification $X^\urcorner$ de $X(\C)$ dans \cite{PoineauCompactificationHybride} où $X(\C)$ se plonge en tant qu'ouvert dense. Le bord de cette compactification, noté $\delta X$, étant un quotient (par la relation d'équivalence de normes) d'un sous-ensemble d'un espace de Berkovich, il faut tout d'abord regarder l'action sur l'analytifié de $X$ au sens de Berkovich (\cite{BerkovichLivre}). L'action de $G$ se prolonge naturellement et M. Maculan a étudié le prolongement de cette action (\cite{MaculanGIT}). Néanmoins, le bord n'est défini que par un sous-ensemble de l'analytifié et il n'est pas assuré que l'action de $G$ préserve ce sous-ensemble. Ainsi, il peut arriver que l'action de $G$ ne soit pas bien définie sur tout le bord de la compactification. 

Le premier objectif est d'étudier l'action de $G$ sur cette compactification $X^\urcorner$ en caractérisant le lieu où l'action de $G$ est bien définie. Caractériser le lieu où l'action de $G$ est bien définie signifie que l'on souhaite déterminer le lieu des $x\in X^\urcorner$ où pour tout élément $g\in G^{an}$ \footnote{Bien que la notation ne le laisse pas apparaître, $G^{an}$ dépends du corps résiduel de $x$.}, $g\cdot x$ définisse un point de $X^\urcorner$.

Une fois le lieu où l'action n'est pas bien définie retiré, on souhaite regarder le quotient de $X^\urcorner$ par l'action de $G$ et le comparer à la compactification $(X/G)^\urcorner$ du schéma quotient comme défini dans \cite{GIT}. On obtient alors deux compactifications homéomorphes de $(X/G)(\C)$. Cela permet d'interpréter le bord de $(X/G)^\urcorner$ comme étant un espace d'orbites non-archimédiennes.

Finalement, on applique ces résultats au cas des fractions rationnelles. On observe alors que cette compactification préserve l'application itération qui est une exigence dynamique que se doit de posséder une compactification des applications rationnelles.

\bigskip 

La construction de la compactification de J. Poineau \cite{PoineauCompactificationHybride}, qui est le cadre de cet article repose sur les espaces hybrides. L'une des premières introductions de ces espaces peut être celle de J. Morgan - P. Shalen (\cite{MorganShalenIntroEspacesHybrides}) qui s'intéressaient déjà à des phénomènes de compactification. Ensuite, V. Berkovich a formalisé les espaces hybrides, en donnant un formalisme d'espaces analytiques sur un anneau de Banach (\cite{BerkovichLivre}). Cette compactification existe dans un cadre plus général que le cas de variétés sur $\C$ : elle existe pour toute variété sur un corps non-trivialement valué $k$. Dans ce cas, la compactification est une compactification de $X^{an}$, l'analytifié de $X$ au sens de Berkovich. L'action de $G$ se prolonge naturellement via le morphisme $\Phi^{an} : (X\times G)^{an} \rightarrow X^{an}$. De plus, sur chaque point de $x \in X^{an}$, il y a une action de $G^{an}_{\h(x)} := pr_2((\Phi^{an})^{-1}(x))$.

L'idée de cette compactification $X^\urcorner$ pour $X$ une variété sur $k$ un corps valué est d'analytifier $X^{hyb}$ selon la norme hybride sur $k$, une norme faisant intervenir  la valeur absolue triviale et la valeur absolue de $k$. La partie provenant de l'analytification sur $k$ muni de la valeur absolue triviale correspond au bord de la compactification. Dans le cas de $\C$, on retrouve donc un bord de nature non-archimédienne (sur $\C$ trivialement valué) et $X(\C)$ est un ouvert dense de $X^\urcorner$.

\bigskip

Nous pouvons maintenant présenter formellement les résultats de ce texte.

Dans le cas où $X$ est affine ou $X$ est le lieu stable d'une $k$-variété, les quotients géométriques de schémas existent par les techniques de GIT \cite{GIT}. On peut donc comparer la compactification de $X/G$ notée $(X/G)\urcorner$ et le quotient de la compactification $X^\urcorner$. Notons que l'on doit nécessairement retirer une partie, l'action n'étant pas bien définie sur tout le bord. Cela donne les deux résultats principaux de ce texte :

Dans un premier temps, on s'intéresse au lieu où l'action est bien définie. 

\begin{thm}(infra \Cref{equivalence dégénérecence quotient et action bien déf})

Soit $k$ un corps non-trivialement valué. Supposons que l'on est dans l'un des deux cas suivants :
\begin{itemize}
\item $X$ est un $k$-schéma affine de type fini, $G$ un groupe algébrique réductif tel que l'action est fermée,
\item $X$ est le lieu stable au sens de GIT d'une $k$-variété $\mathcal{X}$.
\end{itemize}

Soit $x_n \in (X^{an})^\N$ et notons $\pi^{an} : X^{an} \rightarrow (X/G)^{an}$ la projection où l'analytification est selon la valeur absolue de $k$. Supposons que $x_n \rightarrow x \in X^\urcorner$ avec $x\in \delta X$, alors
\begin{center}
l'action de $G_{\mathcal{H}(x)}^{an}$ est bien définie en $x \iff \pi^{an}(x_n) \rightarrow \infty$
\end{center} 
où $\pi^{an}(x_n) \rightarrow \infty$ signifie que cette suite n'a pas de valeur d'adhérence dans $(X/G)^{an}$.
\end{thm}
Ceci permet de caractériser le lieu où l'action de $G^{an}_{\h(x)}$ n'est pas bien définie. On notera $\mathcal{B}$ l'ensemble de ces points. On peut alors définir une relation d'équivalence $\mathcal{G}$ sur $X^\urcorner \backslash \mathcal{B}$. Soient $x, y\in X^\urcorner \backslash \mathcal{B}$ alors $x \mathcal{G} y$ si et seulement si $\exists g\in G^{an}_{\h(x)}, y = g\cdot x$. Cette relation d'équivalence correspond donc à la relation classique sur $X^{an}$ et la prolonge.

On souhaite maintenant comparer le quotient de la compactification par cette relation d'équivalence $\mathcal{G}$ et la compactification du quotient $(X/G)^\urcorner$.

%
%

\begin{thm}(infra \Cref{Si le morphisme est equidimensionnel compactifications sont homeo})

Soit $k$ un corps non-trivialement valué. Supposons que l'on est dans l'un des deux cas suivants :
\begin{itemize}
\item $X$ est un $k$-schéma affine, intègre de type fini, $G$ un groupe algébrique réductif tel que l'action est fermée,
\item $X$ est le lieu stable au sens de GIT d'une $k$-variété intègre $\mathcal{X}$.
\end{itemize} 
Alors, $\mathcal{B}$ est fermé. 

De plus, l'application induite :
\begin{center}
$(X^\urcorner \backslash \lbrace x\in X^\urcorner,$ l'action de $G^{an}_{\mathcal{H}(x)}$ n'est pas bien définie$\rbrace)/\mathcal{G} \rightarrow (X/G)^\urcorner$
\end{center}
est un homéomorphisme qui se restreint en l'identité sur $(X/G)^{an}$.
\end{thm}

Ainsi, si $\mathcal{X}$ est une $k$-variété intègre sur lequel un groupe $G$ algébrique, réductif agit, on peut compactifier le quotient schématique d'un lieu stable de $\mathcal{X}$, noté $\mathcal{X}^s$. Une des façons usuelles de compactifier ce schéma est de regarder le quotient catégorique du lieu semi-stable $\mathcal{X}^s \subset \mathcal{X}^{ss}$. Dans le cas où $\mathcal{X}$ est propre et le lieu stable est défini à partir d'un faisceau inversible ample, alors $\mathcal{X}^{ss}/\! /G$ est une compactification de $\mathcal{X}^s/G$. En utilisant cette compactification $\mathcal{X}^{ss}/\! /G$, le bord que l'on ajoute à $\mathcal{X}^s/G$ peut-être vu comme l'espace topologique $\mathcal{X}^{ss}\backslash \mathcal{X}^s$ quotienté par la relation d'équivalence liant 2 points $x, y$ si et seulement si $\overline{G\cdot x}\cap \overline{G\cdot y} \neq \emptyset$ où $\overline{G\cdot x}$ désigne l'adhérence de l'orbite de $x$. Ici, le bord de la compactification de $(\mathcal{X}^s/G)^\urcorner$ est simplement constitué des orbites des points de $\delta(\mathcal{X}^s) \backslash \mathcal{B}$ ce qui donne une expression plus concrète du bord de la compactification du quotient.

\bigskip

Il est notable que d'utiliser les espaces hybrides pour compactifier $X^{an}$ entraine des complications. Une complication est que dans le cas où le corps n'est pas dénombrable, les espaces hybrides ne sont pas métrisables et donc $X^\urcorner$ ne l'est pas. De plus, on ne sait pas si les espaces hybrides sont ou non angéliques i.e. les compacts sont exactement les ensembles séquentiellement compacts. Poineau \cite{PoineauAngelique} a montré que les espaces de Berkovich sur un corps sont angéliques, T. Lemanissier \cite{LemanissierComPerso} a montré que $\ana^{1,hyb}_\C$ est angélique et C. Gong \cite{GongPersonalCommunication} a montré que les espaces hybrides sur $\C$ étaient bien angéliques en dimension 1 et 2. 

Néanmoins, ils sont particulièrement adaptés pour l'étude de situations mélangeant des aspects archimédiens et non-archimédiens. L'un des cas particulièrement intéressant pour les espaces hybrides est le cas de dégénérescence de phénomènes de nature archimédienne vers des phénomènes de nature non-archimédienne. 

C'est par exemple le cas des fractions rationnelles. La première utilisation des espaces hybrides dans le cadre de la dynamique holomorphe est celle de C. Favre \cite{FavreDegenerationInHybridSpace} qui étudie la convergence de mesures sur $\C$ vers une mesure de nature non-archimédienne. Plus récemment, C. Favre-C. Gong \cite{FavreGong} ont étudié des dégénérescences de fractions rationnelles et ont construit des fractions rationnelles limites définies sur un corps non-archimédien dont ils étudient la dynamique. 

Formellement, on pose $\Rat_d$ l'ensemble des fractions rationnelles de degré $d$, c'est à dire :

\begin{center}
$\mathrm{Rat}_d(\C) = \lbrace f= \frac{P}{Q}, P,Q\in \C[T]$ tel que $P,Q$ soient sans zéros communs et $\mathrm{max}(\deg~P, \deg~Q) = d\rbrace$.
\end{center} 
On dispose sur $\Rat_d$ d'une action de $\SL_2$ où $\SL_2$ agit par conjugaison et on note $\M_d =  \Rat_d/SL_2$ l'espace quotient. On peut définir le résultant d'une fraction rationnelle : tout d'abord, on prend $f = \frac{P}{Q}$ avec $P = \sum_i a_i z^i, Q = \sum_i b_i z^i$. On peut ensuite définir un résultant indépendant du choix de $P,Q$ avec $\mathrm{Res}_f = \vert \frac{\Res(P,Q)}{\max (|a_i|, |b_i|)^{2d}}\vert$. De même, on peut définir le résultant de $\mathfrak{f} \in \M_d$ par $\res_\mathfrak{f} = \max_{f\in \mathrm{Rat}_d, [f] = \mathfrak{f}} \mathrm{Res}_f$.

On peut faire toutes ces constructions pour n'importe quel corps valué $k$ et si $k$ est non-archimédien, on dit que $f$ a bonne réduction si $\mathrm{Res}_f = 1$ et que $f$ a potentielle bonne réduction si $\res_f = 1$. Cela revient à dire que si l'on écrit $f = \frac{a_0z^d + \cdots + a_d}{b_0z^d + \cdots + b_d}$ avec $\mathrm{max}(|a_i|, |b_i|) =1$, alors $f$ a bonne réduction ssi $f$ induit une fraction rationnelle de degré exactement $d$ sur $\tilde{k}$ le corps résiduel de $k$. De même, $f$ a potentielle bonne réduction ssi il existe $M \in \SL_2(\overline{k})$ tel que $f^M$ ait bonne réduction où $\overline{k}$ désigne la clôture algébrique de $k$.

On dit qu'une suite $f_n \in \Rat_d(\C)$ dégénère si la suite ne reste contenue dans aucun compact de $\Rat_d(\C)$ et de même pour $\mathfrak{f_n} \in \M_d(\C)$. L'un des résultats de Favre-Gong s'énonce ainsi. Ils fixent une suite $\mathfrak{f_n}$ qui dégénère et prennent $f_n$ des relevés de $\mathfrak{f_n}$ tel que $\Res_{f_n} = \res_\mathfrak{f_n}$. Alors, en utilisant les espaces de Berkovich, ils construisent pour chaque $\omega \in \beta \N$ où $\beta \N$ est la compactification de Stone-Čech de $\N$, une fraction rationnelle $f_\omega$. Si $\omega$ est l'ultra filtre principal engendré par $n$, alors $f_\omega = f_n$ sinon, c'est une fraction rationnelle définie sur un corps non-archimédien qui s'interprète comme une limite d'une sous-suite des $f_n$. Alors, ils démontrent que si $\omega$ n'est pas un ultra-filtre principal, $f_\omega$ n'a pas potentielle bonne réduction. 

Ces phénomènes de dégénérescence de fractions rationnelles ont déjà été étudiés : tout d'abord par J. Kiwi (\cite{KiwiPolynomialDynamics}) et L. DeMarco-C. McMullen \cite{DeMarcoMcMullenTreesDynamicsPolynomials} dans le cas des polynômes, puis L. DeMarco-X. Faber \cite{DeMarcoFaber} puis plus récemment, par Y. Luo (\cite{LuoLimitRationalMap}, \cite{LuoTreesBarycentricExtension}) qui construit une fraction rationnelle limite à l'aide de techniques hyperboliques. 

La construction de Poineau et les résultats présentés dans ce papier permettent de retrouver des résultats semblables mais dans un contexte différent : Luo, Favre-Gong fixent une suite de fractions rationnelles qui dégénère et construisent des fractions rationnelles limites puis étudient leur dynamique. Dans ce texte, on se rapprochera des techniques de Favre-Gong en étudiant ces aspects via les espaces de Berkovich et non des techniques hyperboliques. De plus, nous considérons une approche plus globale en prenant une compactification de $\Rat_d$ tout entier. Le bord peut-être interprété comme étant des fractions rationnelles définies sur des corps non-archimédiens et nous pouvons regarder la dynamique du bord.  

Ces idées de compactifier l'espace des fractions rationnelles ont déjà été regardées. On peut tout d'abord compactifier en utilisant les outils de D. Mumford - J. Fogarty - F. Kirwan dans \cite{GIT} de la Théorie Géométrique des invariants (GIT). J. Silverman a notamment montré (\cite{SilvermanSpaceRationalMaps}) que compactifier $\M_2$ selon GIT redonnait simplement $\p^2_\C$ mais DeMarco (\cite{DeMarcoSpaceQuadraticRationalMaps}) a montré que cette compactification ne répondait pas aux nécessités dynamiques : par exemple, l'application itération n'y est pas bien définie. Elle réussit a construire deux compactifications homéomorphes de $\M_2$ où l'application itération est bien définie. Mais les deux compactifications ne sont plus homéomorphes pour $d\geq 5$ et suivant la compactification choisie, on peut perdre soit la définition de l'itération soit ne plus avoir de mesures d'équilibre pour les fractions rationnelles du bord.

La compactification hybride permet de surmonter ces difficultés-ci en tout degré. 

En application des résultats de ce texte aux fractions rationnelles, on retrouve tout d'abord un résultat de Favre-Gong \cite{FavreGong} dans un contexte plus général. Dans ce contexte, le fait que l'action de $\SL_2^{an}$ soit bien définie en une fraction rationnelle du bord est équivalent au fait d'avoir potentielle bonne réduction. On obtient ainsi le résultat suivant :

\begin{prop}(infra proposition \ref{pot bonne reduc ssi degenere pas quotient})

Soient $f_n \in \mathrm{Rat}_d^{an}$ où l'analytification est prise au sens de la valeur absolue usuelle sur $k$ telles que $f_n \rightarrow f \in \mathrm{Rat}_d^\urcorner$. Notons $\pi^{an} : \Rat_d^{an} \rightarrow \M_d^{an}$ la projection, alors
\begin{center}
L'action de $\SL^{an}_{2,\mathcal{H}(f)}$ est bien définie $\iff f$ n'a pas potentielle bonne réduction$ \iff \pi^{an}(f_n) \rightarrow \infty$.
\end{center} 
\end{prop}

Ainsi, le comportement dynamique des fractions rationnelles du bord est bien celui attendu : $f \in \delta \Rat_d$ n'a pas potentielle bonne réduction si elle est limite de fractions rationnelles dont les projections sur $\M_d^{an}$ dégénèrent. Une différence entre cette limite et celle obtenue par Favre-Gong \cite{FavreGong} et Luo \cite{LuoLimitRationalMap} est que son corps résiduel est un corps plus petit que dans leurs travaux. Ici, le corps résiduel sur lequel $f$ est défini est un corps de degré de transcendance topologique au plus $2d-1$ sur $\C$. Alors que le corps obtenue par Favre-Gong ou Luo est un corps de degré de transcendance topologique infini. Le fait d'avoir un corps plus petit et en particulier de degré de transcendance fini peut-être très utile comme montré par C. Gong \cite{GongMultiplierScale}.

On peut également exprimer $M_d^\urcorner$ comme un quotient d'un ouvert de $\Rat_d^\urcorner$.

\begin{prop}(infra proposition \ref{bijection entre Rat_d quot et M_d})

L'ensemble $\lbrace f\in \delta \Rat_d,$ f a potentielle bonne réduction$\rbrace$ est un fermé de $\Rat_d^\urcorner$.

On dispose d'un homéomorphisme :
\begin{center}
$(\Rat_d^\urcorner \backslash \lbrace f\in \delta \Rat_d,$ f a potentielle bonne réduction$\rbrace)/\SL_2 \rightarrow \M_d^\urcorner$
\end{center}
qui est l'identité sur $\M_d^{an}$.
\end{prop}

Finalement, l'application itération est bien définie avec cette compactification :

\begin{prop}(infra corollaire \ref{existence itération compactification})

Soit $l \in \N^*$, alors l'application itération $I_l : \M_d \rightarrow \M_{d^l}$ s'étend à $\M_d^\urcorner \rightarrow \M_{d^l}^\urcorner$. 
\end{prop}

De plus, Poineau \cite{PoineauDynamique} a montré que l'on disposait d'une continuité de famille de mesures d'équilibre dans un contexte qui surpasse les espaces hybrides et cela induit une continuité des mesures de probabilités $\mu_f$ pour $f \in \Rat_d$ dans le contexte des espaces hybrides. En particulier, la famille de mesures d'équilibre est une famille continue sur $\Rat_d^\urcorner$. 

\begin{center}
\textbf{Organisation du texte}
\end{center}

Dans la section \ref{sectionEspacesBerko}, on redonne la définition des espaces de Berkovich et particulièrement des espaces hybrides puis on redonne les principales étapes de la construction $X^\urcorner$ de la compactification d'une variété $X$ sur un corps $k$ ainsi que quelques propriétés de cette dernière. On conclut cette partie en redonnant différentes définitions de valuations divisorielles et en réexposant leurs différentes équivalences. Puis, on redonne le résultat connu que les valuations divisorielles forment un ensemble dense des espaces de Berkovich. Finalement, on les utilise pour montrer que si un morphisme de schémas est surjectif, son analytification $f^\beth : X^\beth \rightarrow Y^\beth$ reste surjective. Ensuite, dans la section \ref{SectionConstructionSuite}, on s'intéresse tout d'abord à des questions de continuité puis on cherche à construire des suites d'éléments de $G$ tel que si $x_n \rightarrow x \in X^\urcorner$ et $g\in G^{an}_{\mathcal{H}(x)}$ alors $(x_n, g_n) \rightarrow (x,g)$. Les espaces hybrides n'étant pas angéliques en général, il n'est pas garanti que de telles suites existent et cette section donne une construction explicite de ces suites. Dans la section \ref{SectionResultatActionBord}, on montre les principaux résultats de cet article, en caractérisant le lieu de bonne définition de l'action et en comparant le quotient du compactifié et la compactification du quotient. Finalement, dans la section \ref{ApplicationAuxFractionsRationnelles} on applique ces résultats aux fractions rationnelles.

\begin{center}
\textbf{Convention}
\end{center}

\begin{itemize}
\item Soit $k$ un corps, on notera $\overline{k}$ la clôture algébrique de $k$. Si, de plus, $k$ est un corps valué, on notera $\hat{k}$ sa complétion induite par sa valeur absolue.

\item Soit $X$ un schéma sur un corps valué $k$ localement de type fini. Dans toute la suite, on notera $X^{an}$ pour parler de l'analytification de $X$ selon la valeur absolue de $k$ et on notera $X^{hyb}$ l'analytification de $X$ selon la valeur absolue hybride sur $k$.

\item Une variété sur un corps $k$ est un schéma \textit{séparé}, de type fini sur $k$ (donc quasi-compact). En particulier, tous les schémas affines de type fini sur $k$ seront des variétés.

\end{itemize}

\begin{center}
\textbf{Remerciements}
\end{center}

Je remercie chaleureusement Jérôme Poineau pour nos nombreuses discussions tout au long de la création de cet article, ses conseils et idées et pour sa relecture. Je remercie également Charles Favre pour une discussion intéressante ayant entrainé certaines idées de cet article et ses commentaires ainsi que Chen Gong pour ses commentaires. 

\section{Espaces de Berkovich et compactification hybride}\label{sectionEspacesBerko}

Le but de cette section est de présenter la construction d'une compactification hybride de J. Poineau dans son article \cite{PoineauCompactificationHybride}. Cette construction est celle étudiée durant tout le reste de l'article, nous présentons donc ici quelques résultats nécessaire à la lecture. 
\subsection{Espaces de Berkovich sur des anneaux de Banach}

Soit $(A, \vert\vert \cdot \vert\vert)$ un anneau de Banach. Pour ces premières définitions, on reprend les définitions de V. Berkovich \cite{BerkovichLivre}.

On commence par l'analytification de $\ana^n_A$ que l'on note $\ana^{n,an}_A$ qui est l'espace affine de dimension $n$ sur $A$. On va se concentrer sur l'espace topologique sous-jacent bien qu'il soit muni également d'une structure d'espace localement annelé.

\begin{defi} 
On note $\ana^{n,an}$ l'espace affine de dimension $n$ sur $A$.

L'espace sous-jacent est l'ensemble des semi-normes multiplicatives bornées sur $A[T_1, \cdots, T_n]$. Il s'agit donc de l'ensemble des applications :
\begin{equation*}
|\cdot| : A[T_1, \cdots, T_n] \rightarrow \R_{\geq 0}
\end{equation*}
tel que :
\begin{itemize}
\item $|0| = 0$ et $|1| = 1$,
\item $\forall P,Q \in A[T_1, \cdots, T_n], |PQ| = |P| |Q|$,
\item $\forall P,Q \in A[T_1, \cdots, T_n], |P + Q| \leq |P| + |Q|$,
\item $\forall a\in A, |a| \leq \vert\vert a \vert\vert.$
\end{itemize}

On appelle spectre de Berkovich et on le note $\mathcal{M}(A) := \ana^{0,an}_A$. On dispose d'une projection $\pr : \ana^{n,an}_A \rightarrow \mathcal{M}(A)$ induite par l'injection $A \hookrightarrow A[T_1, \cdots, T_n]$.

Si $x \in \ana^{n,an}_A$, on note $|\cdot|_x$ la semi-norme associée. L'anneau $A[T_1, \cdots, T_n]/(\ker~|\cdot|_x)$ étant intègre, on peut regarder son corps de fraction. Comme $|\cdot|_x$ y induit une valeur absolue, on peut regarder sa complétion que l'on note $\mathcal{H}(x).$

On munit également $\ana^{n,an}_A$ de la topologie la plus grossière telle que pour tout $P \in A[T_1, \cdots, T_n]$, les applications 
\begin{center}
$\begin{cases}
\ana^{n,an}_A &\rightarrow \R_{\geq 0},\\
|\cdot|_x &\mapsto |P|_x
\end{cases}
$
\end{center}
soient continues. Muni de cette topologie, $\ana^{n,an}_A$ est Hausdorff et localement compact. De plus, $\mathcal{M}(A)$ est compact. La projection $\pr : \ana^{n,an}_A \rightarrow \mathcal{M}(A)$ est continue. 
\end{defi}

\begin{ex}
Un anneau de Banach que l'on va beaucoup utiliser est celui des corps hybrides. Soit $k$ un corps muni d'une valeur absolue non-triviale $|\cdot|$, alors on définit sur $k$ une norme hybride $|\cdot|_{hyb}$ tel que  
\begin{center}
$|x|_{hyb} = \max(|x|, |x|_{triviale})$.
\end{center}
On obtient ainsi un anneau de Banach. 
\end{ex}

\begin{prop}
Si $X$ est un schéma localement de présentation finie sur $A$ où $A$ est un anneau de base géométrique, ce qui inclut les corps valués et hybrides, alors on peut l'analytifier. C'est un espace $A$-analytique dans le sens de Berkovich que l'on note $X^{an}$.
\end{prop}

\begin{rem}
Pour une définition précise d'anneau de base géométrique on pourra par exemple se référer au livre de T. Lemanissier - J. Poineau (\cite{LemanissierPoineau}, Définition 3.3.8).
\end{rem}

On rappelle comment construire cette analytification en s'appuyant sur le preuve de Lemanissier-Poineau, Théorème 4.1.4 \cite{LemanissierPoineau}.

\begin{itemize}
\item Première étape : Si $X = \ana^n_A$, alors $X^{an} = \ana^{n,an}$.

\item Deuxième étape : Si $X$ est un sous-schéma fermé de $\ana^n_A$, alors $X$ est défini par un idéal $I$ finiment engendré de $\mathcal{O}(\ana^n_A)$ et $X^{an}$ est le sous espace analytique fermé de $\ana^{n,an}_A$ défini par le faisceau d'idéaux engendré par $I$. 

\item Dernière étape : Si $X$ est localement de présentation finie, alors $X = \bigcup U_i$ où les $U_i$ sont des variétés affines de présentation finie que l'on analytifie comme précédemment, ainsi $X^{an}$ est obtenu en recollant les $U_i^{an}$.
\end{itemize}

\begin{prop}
Si $X$ et $Y$ sont deux $A$-schémas localement de présentation finie et $f : X \rightarrow Y$ un morphisme de schéma, alors on peut analytifier $f$ pour avoir $f^{an} : X^{an} \rightarrow Y^{an}$.
\end{prop}

Cette analytification préserve la plupart des propriétés du morphisme de schémas. Dans le cas, où l'on dispose d'un morphisme fini, on a de plus le résultat suivant : 




\begin{lem}\label{morphisme fini sur meme dimension est ouvert} Lemme 3.2.4 de \cite{BerkovichLivre}

Soit $\phi : X \rightarrow Y$ un morphisme fini d'espaces $k-$analytiques tel que $\dim(X) = \dim(Y)$ et $X$ est localement irréductible. Alors $\phi$ est un morphisme ouvert.
\end{lem}




\subsection{Espaces de Berkovich hybrides}

On va maintenant se focaliser au cas où l'anneau de Banach $A$ est un corps muni d'une norme hybride. On présente la section 2 de l'article de Poineau \cite{PoineauCompactificationHybride}, on omet les preuves mais on rappelle les différentes définitions.

\begin{prop}
Soit $(k, |\cdot|)$ un corps valué, et on note $k_{hyb}$ le corps muni de la norme hybride. Alors le spectre de Berkovich est :
\begin{center}
$\mathcal{M}(k_{hyb}) = [0,1]$,
\end{center} 
où l'identification vient de l'association à tout $0\leq \varepsilon \leq 1$ de la valeur absolue $|\cdot|^\varepsilon$ et $|\cdot|^0$ correspond à la valeur absolue triviale.

Ainsi, les corps résiduels $\mathcal{H}(\varepsilon)$ sont les complétés de $k$ muni de la valeur absolue $|\cdot|^\varepsilon$. On les note $\hat{k}_\varepsilon$.   

Donc, si $X$ est un espace $k_{hyb}$-analytique, il est muni d'une projection $\pr : X \rightarrow \mathcal{M}(k_{hyb})$ et pour tout $\varepsilon \in [0,1], \pr^{-1}(\varepsilon) =: X_\varepsilon$ est un espace $\mathcal{H}(\varepsilon)$-analytique. 
\end{prop} 

\begin{rem}
Si $X$ est un $k_{hyb}$ espace analytique, pour $\varepsilon > 0, X_\varepsilon$ ont des espaces topologiques sous-jacent isomorphes : tous les espaces $\mathcal{H}(\varepsilon)$-analytique étant tous des espaces $k$-analytiques muni d'une normalisation différente. Et pour $\varepsilon = 0, X_0$ est un espace analytique sur un corps trivialement valué.

Dans le cas où $k$ est un corps archimédien alors pour $\epsilon > 0$, on dispose d'espaces analytiques complexes et pour $\varepsilon = 0$ on trouve un espace de Berkovich de nature non-archimédienne. 
Ainsi, les corps hybrides peuvent permettre de lier des phénomènes archimédiens et non-archimédiens.   
\end{rem}

Poineau introduit la notion de flot qui permettra de définir une relation d'équivalence nécessaire à la construction d'une compactification hybride. On présente ici sa définition.

\begin{defi}
Soit $\varepsilon \in [0,1]$, alors on définit :
\begin{center}
$I_\varepsilon :=
		\begin{cases}
        [0, +\infty[ & \mathrm{si}~ \varepsilon = 0\\
        [0, \frac{1}{\varepsilon}] & \mathrm{sinon.}    
        \end{cases}$
\end{center}
On notera $I^*_\varepsilon := I_\varepsilon \backslash \lbrace 0 \rbrace$.

De plus, si $S$ est un espace $k_{hyb}$-analytique, alors avec la projection $\pr : S \rightarrow \mathcal{M}(k_{hyb})$, pour tout $x\in S$, on définit $I_x := I_{\pr(x)}$.
\end{defi}

\begin{lem}
Soit $x \in \ana^{n,an}_{k_{hyb}}$ et $\alpha \in I_x$, alors l'application :
\begin{center}
$P \in k[T_1, \cdots, T_n] \mapsto |P(x)|^\alpha \in \R_{\geq 0}$ 
\end{center}
définit un point de $\ana^{n,an}_{k_{hyb}}$ que l'on note $x^\alpha$. On a $\pr(x^\alpha) = \alpha \pr(x)$.

De plus, si $\alpha \in I^*_x$, alors les corps $\mathcal{H}(x)$ et $\mathcal{H}(x^\alpha)$ sont isomorphes.
\end{lem}

On peut désormais définir le flot :

\begin{defi}
Posons 
\begin{center}
$D(\ana^{n,an}_{k_{hyb}}) := \bigcup_{x \in \ana^{n,an}_{k_{hyb}}} \lbrace x\rbrace \times I_x^\alpha \subset \ana^{n,an}_{k_{hyb}} \times \R_{> 0}$.
\end{center}
Le flot est alors l'application :
\begin{center}
$\Phi : \begin{cases} 
		D(\ana^{n,an}_{k_{hyb}}) &\rightarrow \ana^{n,an}_{k_{hyb}},\\
		(x,\alpha)&\mapsto x^\alpha.
		\end{cases}$
\end{center}
\end{defi}

\begin{prop}
Le flot est une application continue et ouverte.
\end{prop}
Pour la preuve, on pourra se référer à la proposition 2.10 de Poineau \cite{PoineauCompactificationHybride}.

On peut également définir la notion de trajectoire d'un point et d'un ensemble.

\begin{defi}
Soit $x \in \ana^{n,an}_{k_{hyb}}$, alors la trajectoire du point $x$ est l'ensemble $T(x)$ défini par :
\begin{center}
$T(x) := \Phi(x, I_x^*) = \lbrace x^\alpha, \alpha \in I_x^*\rbrace.$ 
\end{center}
\end{defi}

\begin{rem}
Soit $x\in \ana^{n,an}_{k_{hyb}}$, alors pour $y \in \ana^{n,an}_{k_{hyb}}$ si $y \in T(x)$ alors $T(y) = T(x)$. 
\end{rem}

Ce résultat va permettre de définir une relation d'équivalence en utilisant les trajectoires des points.

On peut de plus définir la trajectoire d'un ensemble. 
\begin{defi}\label{def trajectoire ensemble}
Soit $V$ un sous-ensemble de $\ana^{n,an}_{k_{hyb}}$, alors la trajectoire de $V$ est l'ensemble 
\begin{center}
$T(V) := \bigcup_{x\in V} T(x) \subset \ana^{n,an}_{k_{hyb}}$.
\end{center}

De plus, si $V, V'$ sont deux ensembles de $\ana^{n,an}_{k_{hyb}}$, alors $T(V\cup V') = T(V) \cup T(V')$ et $T(V\cap V') = T(V) \cap T(V')$.
\end{defi}

\subsection{Construction de la compactification hybride}
Dans cette partie, on présente la construction d'une compactification hybride, on se base sur les sections 3 et 4 de l'article de Poineau \cite{PoineauCompactificationHybride}, on omet les preuves mais l'on présente les différents résultats.

Tout d'abord, la construction ne se fait que sur un ouvert de $X^{hyb}$ où $X$ est une variété sur $k$ et $X^{hyb}$ signifie que l'on analytifie $X$ sur $k_{hyb}$. L'objectif est de retirer de $X^{hyb}$ une fibre générique.

M. Raynaud (\cite{RaynaudGeometrieRigide}), P. Berthelot (\cite{BerthelotLivre}) , V. Berkovich (\cite{BerkovichSchemaFormel1}, \cite{BerkovichSchemaFormel2}) puis A. Thuillier (\cite{ThuillierGeometrieToroidale}) ont remarqué que les espaces non-archimédiens peuvent être utilisés pour définir une notion de fibre générique pour des schémas formels. Comme dans la section 3 de Poineau \cite{PoineauCompactificationHybride}, on présente la construction de Thuillier.

On prend $\mathcal{X}$ un schéma formel sur $k_0$, on rappelle que cela signifie que l'on prend $k$ trivialement valué, qui est localement algébrique. A ce schéma formel, on associe une fibre générique $\mathcal{X}^\beth$ qui est un espace $k_0$ analytique et une application $r_{\mathcal{X}} : \mathcal{X}^\beth \rightarrow \mathcal{X}_s$ qui est anti-continue i.e. l'image réciproque d'un ouvert est fermé.

On ne présente la construction que dans le cas affine, mais elle existe dans un cadre plus général.

Soit $\mathcal{X} = X$ une variété affine, $X = \Spec(A)$. Alors,
\begin{center}
$X^\beth = \mathcal{M}(A)$
\end{center}
où $A$ est trivialement valué. L'application $r_X : \mathcal{M}(A) \rightarrow \Spec(A)$ est l'application de réduction telle que pour $x\in \mathcal{M}(A),$
\begin{center}
$r_X(x) = \lbrace a\in A, |a(x)[ <1 \rbrace$.
\end{center}

Ceci permet de définir la partie du bord de la compactification hybride.

\begin{defi}
Soit $X$ une variété sur $k$, alors on pose 
\begin{center}
$X_\infty := X^{an}_0 \backslash X^\beth$
\end{center}
c'est un ouvert de $X^{an}_0$ et donc c'est un espace $k_0$-analytique.
\end{defi}

On peut regarder quelques exemples.
\begin{ex}\label{exemple de X infty}
Si $X = \ana^1_k$, on note $\eta_{a,r} \in \ana^{1,an}_{k_0}$ la semi-norme $P = \sum a_k(T-a)^k \mapsto \mathrm{max}|a_k|_0 r^k$. Comme $k_0$ est trivialement valué, $\eta_{a,r} \leq 1 \iff \eta_{a,r}(T-a) \leq 1 \iff \eta_{a,r}(T) \leq 1$. Ainsi,

\begin{center}
$(\ana^1_k)_\infty = \lbrace x\in \ana^{1,an}_{k_0}, \exists P\in k[T] |P(x)| >1\rbrace = \lbrace \eta_{0,r}, r\in \R_{>1}\rbrace$.
\end{center}

De même, si $X = \mathbf{G}_{m,k}$ alors

\begin{align*}
(\mathbb{G}_{m,k})_\infty &= \lbrace x\in \mathbb{G}^{an}_{m,k_0}, \exists P\in k[T, T^{-1}] |P(x)| >1\rbrace \\
&= \lbrace x\in \mathbb{G}^{an}_{m,k_0}, \mathrm{max} (|T(x)|, |T^{-1}(x)) >1\rbrace \\
&= \lbrace \eta_{0,r}, r\in \R_{>0}, r\neq 1\rbrace.
\end{align*}
\end{ex}

Maintenant que l'on a défini la partie "bord" de la compactification, on peut définir l'objet à quotienter pour avoir une compactification.

\begin{defi}
Soit $X$ une $k$ variété, alors on pose 
\begin{center}
$X^+ := X^{hyb} \backslash X^\beth$.
\end{center}
C'est un ouvert de $X^{hyb}$ et c'est donc un espace $k_{hyb}$-analytique. On peut remarquer que $X^+_0 = X_\infty$.
\end{defi}

On dispose de quelques résultats sur les morphismes.

\begin{prop}
Soit $X, Y$ deux $k$ variétés et $f : X \rightarrow Y$ un morphisme propre, alors l'analytifé $f^{hyb} : X^{hyb} \rightarrow Y^{hyb}$ est propre et se restreint en un morphisme $f^+ : X^+ \rightarrow Y^+$.
\end{prop}

Pour la preuve, on pourra se référer à la proposition 4.2 de Poineau \cite{PoineauCompactificationHybride}.

\begin{lem}Proposition 4.6 de \cite{PoineauCompactificationHybride}. \label{morphisme analytifié garde les propriétés}

Soient $X,Y$ deux $k$-schémas de type fini et soit $f : X \rightarrow Y$ un morphisme plat, fini alors $f^{hyb} : X^{hyb} \rightarrow Y^{hyb}$ l'est aussi. De plus, si $f$ est propre alors $f^+$ est un morphisme plat, fini également.
\end{lem}

On a de même des propriétés sur les variétés qui restent vraies dans le cas des espaces hybrides.

\begin{prop} Proposition 4.5 de \cite{PoineauCompactificationHybride} \label{normalité se préserve par espaces hybrides}

Soit $X$ une $k$-variété. Alors si $X$ est normal, $X^{hyb}$ et $X^+$ le sont aussi.
\end{prop}

\begin{proof}On redonne la preuve donnée par Poineau.

Il suffit de le montrer pour $X^{hyb}$, comme $X^+$ est un ouvert de $X^{hyb}$. 

Soit $x\in X^{hyb}$, on note $\epsilon(x) := \pr(x)$. Alors, $\mathcal{O}_{X_{\epsilon(x)},x}$ est normal. Dans le cas, où $X_{\epsilon(x)}$ est un espace analytique complexe, on peut se référer à \cite{SGA}, Exposé XII, Proposition 2.1 et dans le cas où $X_{\epsilon(x)}$ est un espace de Berkovich, on peut se référer à \cite{DucrosEspacesBerkoExcellents}, Théorème 3.4. 

Par la section 0.5.1 de \cite{DucrosEspacesBerkoExcellents} et les références dans cette section, la propriété de normalité de l'anneau locale $\mathcal{O}_{X,x}$ se vérifient après des extensions fidèlement plates.

Or, par le Théorème 4.3 de \cite{Berger}, le morphisme
\begin{center}
$\mathcal{O}_{X,x} \rightarrow \mathcal{O}_{X_{\epsilon(x)},x}$
\end{center}
est plat.
\end{proof}
Cela permet d'avoir un équivalent au lemme 3.2.4 de \cite{BerkovichLivre} (voir lemme \ref{morphisme fini sur meme dimension est ouvert}) dans le cadre hybride.

\begin{lem}\label{morphisme fini sur espace hybride est ouvert}
Soient $X, Y$ deux $k$-schémas de type fini de même dimension avec $Y$ normal. Soit $f : X \rightarrow Y$ un morphisme quasi-fini. Alors $f^{hyb} : X^{hyb} \rightarrow Y^{hyb}$ est ouvert.
\end{lem}

\begin{proof}
La preuve s'appuie sur la démonstration du lemme 3.2 de \cite{BanicaStanasila2} et sur une suggestion de J. Poineau.

Par le théorème 5.2.9 de \cite{LemanissierPoineau}, un morphisme quasi-fini est fini en tout point. Comme être un morphisme ouvert est une propriété locale, on peut donc se ramener au cas où $f$ est un morphisme fini.

Soit $x \in X^{hyb}$ et $\mathcal{U}$ un voisinage de $x$. Notons $b := \pr(x) \in [0,1]$.

Il faut montrer que $f^{hyb}(\mathcal{U})$ contient un voisinage de $f^{hyb}(x)$.

On peut supposer que $\overline{\mathcal{U}}$, l'adhérence de $\mathcal{U}$, est compacte et comme les fibres sont finies, on peut également supposer que $\overline{\mathcal{U}} \cap (f^{hyb})^{-1}(f^{hyb}(x)) = \lbrace x\rbrace$. Ainsi, $f^{hyb}(x) \notin f^{hyb}(\delta \mathcal{U})$ où $\delta \mathcal{U}$ désigne la frontière de $\mathcal{U}$.

Soit $V$ un voisinage ouvert de $f^{hyb}(x)$ tel que $V \cap f^{hyb}(\delta \mathcal{U}) = \emptyset$. 

Posons $\mathcal{U}' := \mathcal{U} \cap (f^{hyb})^{-1}(V)$ et $g : \mathcal{U}' \rightarrow V$, le morphisme induit par $f^{hyb}$. Alors, $g$ est fini et $g(\mathcal{U}')$ est un fermé analytique de $V$. Donc, il est défini par un faisceau cohérent d'idéaux $\mathcal{F}$ de $\mathcal{O}_V$.

De même, $g_b : \mathcal{U}'_b \rightarrow V_b$ est fini et $g_b(\mathcal{U}'_b)$ est un fermé analytique de $V_b$ de dimension $n = \dim(X_b) = \dim(Y_b) = \dim(V_b)$. 

Comme $Y$ est normal, alors $Y^{hyb}_b$ est normal par la proposition \ref{normalité se préserve par espaces hybrides}. Donc, $V_b$ est un ouvert normal de $Y^{hyb}_b$.

Ainsi, $g_b(\mathcal{U}'_b)$ contient la composante irréductible de $V_b$ qui contient $f^{hyb}(x)$. Donc il existe un ouvert $f^{hyb}(x) \in W$ de $Y^{hyb}_b$ tel que $W \subset V_b$. Donc, les germes de $\mathcal{F}|_{V_b}$ sont nuls en $f^{hyb}(x)$. Or par le Théorème 4.3 de \cite{Berger}, le morphisme $\mathcal{O}_{V, f^{hyb}(x)} \rightarrow \mathcal{O}_{V_b, f^{hyb}(x)}$ est plat. Ainsi, on a : 
\begin{center}
$0 = (\mathcal{F}|_{V_b})_{f^{hyb}(x)} = \mathcal{O}_{V_b, f^{hyb}(x)} \otimes \mathcal{F}_{f^{hyb}(x)}$
\end{center}
donc $\mathcal{F}_{f^{hyb}(x)} = 0$. Donc il existe $V'$ un ouvert de $V$ contenant $f^{hyb}(x)$ tel que $\mathcal{F}|_{V'}= 0$.

Ainsi, $g(\mathcal{U}') \cap V' = V'$ et $f^{hyb}(x) \in V' \subset f^{hyb}(\mathcal{U})$. 
\end{proof}

Finalement, pour définir la compactification hybride, il reste à définir une relation d'équivalence.

\begin{defi}
Soit $X$ une $k$-variété et soient $x,y \in X^{hyb}$. On dit que $x,y$ sont équivalents par le flot si $T(x) = T(y)$ et on note $x \Phi y$.
\end{defi}

\begin{lem}
Soit $X$ une $k$ variété, alors $\Phi$ est une relation d'équivalence.
\end{lem}
Pour la preuve, on pourra se référer au lemme 4.8 de \cite{PoineauCompactificationHybride}.

On peut finalement définir la compactification :
\begin{defi}\label{immersion de Xan dans compactification}
Soit $X$ une $k$ variété, l'ensemble 
\begin{center}
$X^\urcorner := X^+/\Phi$
\end{center}
est appelé la compactification valuative de $X$ ou compactification hybride. On munit cet ensemble de la topologie quotient, donc en particulier l'image d'un sous-ensemble $V$ de $X^+$ est ouverte ssi $T(V)$ est ouvert dans $X^+$.

On note $q : X^+ \rightarrow X^\urcorner$ l'application quotient. 

On note $\delta X := q(X_\infty)$ le bord de la compactification.

Finalement, on note $i$ l'immersion suivante :
\begin{center}
$i : (X \otimes_k \hat{k})^{an} = X_1^+ \rightarrow X^+ \rightarrow X^\urcorner.$
\end{center} 
\end{defi}

\begin{prop}\label{prop section continue partie archi} Lemme 4.11 et 4.15 de \cite{PoineauCompactificationHybride}.

Soit $k$ un corps muni d'une valeur absolue non triviale. Soit $X$ une $k$ variété, alors on peut analytifier $X$ selon la valeur absolue hybride sur $k$.

Et $i : X^{an} \rightarrow X^\urcorner$ est un homémorphisme.
\end{prop}

\begin{ex}
On a vu à l'exemple \ref{exemple de X infty} que $(A^1_k)_\infty = \lbrace \eta_{0,r}, r>1\rbrace$, donc $\delta (A^1_k)$ n'est qu'un unique point, et de même comme $(\mathbb{G}_{m,k})_\infty = \lbrace \eta_{0,r}, r>0, r\neq 1\rbrace$ alors $\delta (\mathbb{G}_{m,k})$ consiste de deux points. 
\end{ex}

Finalement, $X^\urcorner$ dispose de plusieurs propriétés topologiques.
\begin{prop}
Soit $X$ une $k$-variété. Alors, $X^\urcorner$ est Hausdorff et compact et $X^{an}$ est dense dans $X^\urcorner$. 

De plus, $X^\urcorner$ est localement connexe par arcs et si $X$ est connexe, $X^\urcorner$ est connexe par arcs.

Si $k$ est dénombrable, alors $X^\urcorner$ est métrisable.
\end{prop}
Pour les preuves, on pourra se référer aux propositions 4.16, 4.22 et 4.23, au théorème 4
19 et au lemme 4.20 de \cite{PoineauCompactificationHybride}.

\subsection{Valuations divisorielles}
Dans cette partie, on considère $X$ un schéma de type fini sur un corps $k$ et on va considérer certaines valuations particulières de $X^\beth$ : les valuations divisorielles. Finalement, on utilisera les valuations divisorielles pour montrer que si un morphisme $f: X\rightarrow y$ de $k$-schémas de type finis, intègres est surjectif alors $f^\beth : X^\beth \rightarrow Y^\beth$ est aussi surjectif.

Dans cette partie, on s'appuie sur l'article de M. Vaquié \cite{VaquieValuations}.

On rappelle que dans le cas où $X$ est affine, $X = \Spec(A)$ alors $X^\beth = M(A)$ où $A$ est trivialement valué. 

Les valuations divisorielles sont des valuations dites d'Abhyankar. On va donc redéfinir la notion de valuation d'Abhyankar.

\begin{defi}
Soit $k$ un corps valué et $l$ une extension valuée de $k$. Notons $\tilde{l}$ et $\tilde{k}$ les corps résiduels de $l$ et $k$.

Alors, on note 
\begin{center}
$s(l) := \mathrm{tr}.\deg.(\tilde{l}/\tilde{k})$ et $t(l) := \dim_\Q(|l^\times|^\Q/|k^\times|^\Q).$ 
\end{center}
Soit $X$ un espace $k$-analytique et soit $x \in X$. On note $s(x) = s(\mathcal{H}(x))$ et $t(x) = t(\h(x))$.
\end{defi}

Ces deux quantités sont reliés par l'inégalité d'Abhyankar, on pourra se référer à (\cite{BourbakiAlgCom57}, VI, §10.3, Cor 1).

\begin{thm}
Soit $l$ une extension valuée de $k$, alors 
\begin{equation*}
s(l) + t(l) \leq \mathrm{tr}.\deg.(l/k).
\end{equation*}
En particulier, si $X$ est un schéma de dimension $n$ sur $k$, alors pour tout $x\in X$
\begin{equation*}
s(x) + t(x) \leq n.
\end{equation*}
Les points $x\in X$ vérifiant le cas d'égalité seront appelés point d'Abhyankar.
\end{thm}

Poineau a démontré le résultat suivant dans \cite{PoineauAngelique}, corollaire 4.8.
\begin{prop}
L'ensemble des points d'Abhyankar d'un espace analytique est dense.
\end{prop}

Dans le cas, où $k$ n'est pas trivialement valué, Poineau a montré des résultats plus fort dans ce même article (proposition 4.5, corollaire 5.7). Certains sous-ensemble des points d'Abhyankhar sont denses. 

Dans le cas où $k$ est trivialement valué et $X$ est un schéma de type fini sur $k$, il existe également des sous-ensembles des points d'Abhyankar qui sont denses.

\begin{prop}\label{points de type 2 sur type 3 sont denses}
Soit $X$ un $k$-schéma de type fini, de dimension $n \geq 1$. Munissons $k$ de la valeur absolue triviale. Alors,
\begin{center}
$\lbrace x \in X^{an}, t(x) = 1, s(x) = n-1\rbrace$ est dense dans $X^{an}$.
\end{center} 
Dans le cas, où $n=1$ cela correspond simplement à la densité des points d'Abhyankar.
\end{prop}

\begin{proof}
On suppose $n \geq 2$.

On commence par supposer que $X= \ana^n_k$. Soit $U$ un ouvert non-vide de $X^{an}$. Quitte à restreindre $U$, on peut supposer que $U$ est connexe.

Notons alors $\pi_1$ la projection sur la première coordonnée. C'est un morphisme ouvert. Alors $\pi_1(U)$ est un ensemble non vide connexe de $\ana^{1,an}_k$. Les seuls points $y\in \ana^{1,an}_k$ ne vérifiant pas $t(y) =1$ sont les $\hat{\overline{k}}$-points et le point de Gauss. Donc, si $\pi_1(U)$ ne contient aucun point vérifiant $t(y)=1$, c'est un unique point puisque que c'est un ensemble connexe. Alors $\pi_1(U)$ est fermé, ce qui est absurde par connexité de $\ana^{1,an}_k$. Donc, il existe $y \in \pi_{1}(U)$ avec $t(y) = 1$.

Ainsi $U \cap (\pi_1)^{-1}(y)$ est un ouvert de l'espace $\h(y)$-analytique $\ana^{n-1}_{\h(y)}$ et $\h(y)$ n'est pas trivialement valué puisque $t(y) = 1$, donc quitte à restreindre $U \cap (\pi_1)^{-1}(y)$ on peut supposer que c'est un espace strictement $k$-affinoïde. On peut donc appliquer la proposition 4.5 de \cite{PoineauAngelique} qui donne l'existence d'un $x \in U \cap (\pi^1)^{-1}(y)$ tel que $s(x) = n-1$. Ainsi, $(x,y) \in U$ et $s(x,y)  =n-1, t(x,y) = 1$.

Supposons maintenant que $X$ est un schéma de type fini sur $k$. 

Comme le résultat est un résultat local, on peut supposer que c'est un schéma affine. On dispose alors par la normalisation de Noether d'un morphisme fini $\pi^{noether} : X \rightarrow \ana^n_k$. Son analytifié $\pi^{noether, an}$ est un morphisme d'espace $k$-analytiques tel que $\dim(X) = \dim(\ana^n_k)$ et $\ana^{n,an}_k$ est localement irréductible. Alors on peut appliquer le lemme 3.2.4 de \cite{BerkovichLivre} (voir lemme \ref{morphisme fini sur meme dimension est ouvert}) qui assure que $\pi^{noether, an}$ est ouverte. Comme le morphisme $\pi^{noether, an}$ est fini, $s$ et $t$ sont invariants par ce morphisme. Donc la densité de $\lbrace x \in \ana^{n,an}_k, t(x) = 1, s(x) = n-1\rbrace$ dans $\ana^{n,an}$ permet de retrouver celle de $\lbrace x \in X^{an}, t(x) = 1, s(x) = n-1\rbrace$ dans $X^{an}$.     
\end{proof}

On peut alors définir les valuations divisorielles. On donne la définition de M.Vaquié.

\begin{defi}
Soit $A$ un anneau intègre, de type fini sur un corps $k$, de dimension $n$ et de corps de fraction $K$. Une valuation sur $K$ positive sur $A$ est dite divisorielle au sens de Vaquié si elle vérifie :
\begin{center}
rang~$v$ = 1 et $\mathrm{tr}.\deg.(\tilde{K}/\tilde{k}) = n-1$.
\end{center}
Cela correspond à prendre un élément $x$ de $M(A, |\cdot|_0)$ tel que $t(x) =1$ et $s(x) =n-1$.
\end{defi}

\begin{prop}
Soit $X$ un schéma intègre, de type fini sur un corps $k$ alors les valuations divisorielles au sens de Vaquié sont denses dans $X^\beth$.
\end{prop}

\begin{proof}
Le fait d'être dense étant une propriété locale, on peut supposer que $X$ est affine, $X = Spec(A)$ avec $A$ de dimension $n$. 

Par la proposition \ref{points de type 2 sur type 3 sont denses}, on sait que les points de $X^\beth$ vérifiant $s(x) = n-1$ et $t(x)=1$ sont denses dans $X^\beth$. Montrons qu'ils correspondent à des valuations divisorielles au ses de Vaquié. Soit $x \in X^\beth$ tel que $s(x) = 1$ et $t(x) = n-1$. On note $v$ la valuation associée sur $A$. Comme $s(x) = 1$, on a $rang~v=1$ et comme $t(x) = n-1$, on a $deg.tr.(\tilde{K}/\tilde{k}) = n-1$ où $K$ est le corps de fractions de $A$. Donc c'est une valuation divisorielle au sens de Vaquié.
\end{proof}

Le nom de valuation divisorielle vient de la situation géométrique suivante :

Soit $X$ un schéma affine intègre de type fini sur un corps $k$. Soit $D$ un sous-schéma intègre tel que l'anneau $\mathcal{O}_{X,D}$ soit régulier. On utilise cette notation pour parler de l'anneau local au point générique de $D$. Si le schéma $D$ n'est pas un diviseur, on peut prendre l'éclatement $\pi : E_D \rightarrow X$ de $X$ le long de $D$ puis normaliser $E_D$ pour obtenir un anneau de valuation discrète $\mathcal{O}_{X', \pi^{-1}(D)}$ où $X'$ est le normalisé de $E_D$. Ainsi, on définit une valuation sur $\mathcal{O}(X)$. 

On peut ainsi définir la notion de valuation géométrique.

\begin{defi}
Soit $X$ un schéma intègre, de type fini sur $k$. Soit $Y$ un schéma normal et $E$ un diviseur premier de $Y$. Soit $\pi : Y \rightarrow X$ un morphisme propre, birationnel. Alors l'anneau de valuation $\mathcal{O}_{Y, E}$ induit une valuation sur $\mathcal{O}(X)$ que l'on appelle valuation géométrique divisorielle. 
\end{defi}

\begin{rem}
Toutes ces définitions peuvent se faire dans le cas où $X$ est un $k$-schéma intègre, excellent.
\end{rem}

On souhaite maintenant lier les 2 notions de valuations divisorielles. Pour cela, on s'appuie sur la proposition 6.4 de Vaquié \cite{VaquieValuations}.
\begin{prop} Proposition 6.4 \cite{VaquieValuations}.

Soit $X$ un $k$-schéma intègre, excellent de corps des fonctions $F(X) = K$. Pour toute valuation $v$ de $K$, triviale sur $k$, centrée sur $X$, la dimension du centre de $v$ sur $X$ est inférieure ou égale à $\mathrm{tr}.\deg.(\tilde{K}/\tilde{k})$. De plus, il existe $Z$ un éclatement de $X$ le long d'un sous-schéma fermé tel que le centre de $v$ sur $Z$ est de dimension égale à $\mathrm{tr}.\deg.(\tilde{K}/\tilde{k})$.
\end{prop}

\begin{rem}
La preuve ne donne pas de conditions sur le sous-schéma fermé que l'on éclate. En particulier, on ne peut à priori pas se restreindre aux sous-schémas réduits, irréductibles.
\end{rem}

\begin{prop}\label{Les valuations divisorielles sont denses}
Soit $X$ un schéma intègre, de type fini sur $k$. Alors les valuations divisorielles géométriques correspondent aux valuations divisorielles au sens de Vaquié. De plus,  il suffit de regarder les valuations divisorielles géométriques provenant de la situation où  $\pi : Y \rightarrow X$ est la composée d'une normalisation et d'un éclatement d'un sous-schéma fermé.

En particulier, les valuations divisorielles géométriques sont denses dans $X^\beth$.
\end{prop}

\begin{proof}
Il suffit de vérifier que tout valuation divisorielle au sens de Vaquié correspond à une valuation divisorielle géométrique. On se ramène au cas où $X$ est affine. 

Soit $x \in X^\beth$ une valuation divisorielle au sens de Vaquié. Alors $t(x) =1$ et $s(x) =n-1$ où $\dim(A) =n$. Comme par le lemme 3.4 tout valuation de $X^\beth$ est centrée sur $X$, on sait que $x$ est centrée sur $X$. 

Alors par la proposition 6.4 de \cite{VaquieValuations}, il existe $Z$ éclaté de $X$ tel que le centre de $x$ en $Z$ ait dimension $n-1$ et est donc codimension 1. On peut alors considérer $n : Y \rightarrow Z$ le normalisé de $Z$, comme $Z$ est de Nagata, ce morphisme est fini. Donc l'image réciproque par $n$ du centre de $x$ en $Z$ a codimension 1. On note $D$ une de ses composantes irréductibles ayant codimension 1. Comme l'anneau $Y$ est normal, l'anneau $\mathcal{O}_{Y,D}$ est de valuation discrète. Cette valuation prolonge $x$ sur $\mathcal{O}(X)$ et provient de la situation géométrique $\pi : Y \rightarrow X$ où $\pi$ est la composée de la normalisation et un éclatement.
\end{proof}

\begin{rem}
On pouvait retrouver cette démonstration avec la proposition 10.1 de Vaquié \cite{VaquieValuations} qui s'appuie sur les deux articles de M. Spivakovsky \cite{SpivakovskyValuations}, \cite{SpivakovskyValuations2}. La preuve y est indiquée dans le cas où $X = \Spec~A$ avec $A$ un anneau local, mais l'hypothèse d'anneau local n'est pas nécessaire dans la preuve.
\end{rem}

\begin{rem}
Dans leur article, M. Jonsson et M. Musta\c t\u a \cite{JonssonMustata} ont également montré que les valuations divisorielles étaient denses dans le cas d'un schéma régulier sur un corps de caractéristique 0. 

Ils ont de plus caractérisé toutes les valuations d'Abhyankar. Pour cela, ils définissent des valuations quasi-monomiales. Ce sont des valuations qui sont localement monomiales sur un modèle birationnel de $X$. Plus précisement, si $\pi : Y \rightarrow X$ est un morphisme propre, birationnel avec $Y$ régulier et connexe et $\underline{y} = (y_1, \cdots, y_r)$ est un système de coordonnées algébriques en un point $\eta \in Y$. On peut alors définir une valuation sur $\mathcal{O}_{Y,\eta}$ qui induit donc une valuation sur $\mathcal{O}_X$. Soit $\alpha \in \Z^r$, alors on définit $val_\alpha$. Soit $f\in \mathcal{O}_{Y,\eta}$ qui s'écrit $f = \sum_{\beta \in \Z^r_{\geq 0}} c_\beta y^\beta$ comme élément de $\hat{\mathcal{O}_{Y,\eta}}$ et les $c_\beta$ sont soit nuls soit des unités. 

Alors $\val_\alpha(f) =\min\lbrace \sum \alpha_i\beta_i, c_\beta\neq 0\rbrace.$

Ainsi, les points d'Abhyankar correspondent dans ce contexte à toutes ces valuations et les valuations divisorielles sont celles qui ont un rang égal à 1 parmi celles-ci.
\end{rem}

Dans toute la suite, on parlera uniquement de valuations divisorielles. 

On va maintenant utiliser les valuations divisorielles pour montrer que si un morphisme de schémas $k$-schémas de type fini $f : X \rightarrow Y$ est surjectif, alors la restriction de son analytification $f^\beth : X^\beth \rightarrow Y^\beth$ reste surjective.

\begin{rem}
Il n'est à priori pas clair que cette restriction soit surjective. On sait que $f^{an}_0 : X^{an}_0 \rightarrow Y^{an}_0$ est surjective mais rien n'oblige à priori d'avoir $\forall y \in Y^\beth, (f^{an})^{-1}(y) \cap X^\beth \neq \emptyset$. En général, si $f$ n'est pas surjectif et $y \in Y^\beth \cap f^{an}(X^{an}_0)$, on peut avoir $(f^{an})^{-1}(y) \cap X^\beth = \emptyset$. On peut par exemple prendre $f : \Spec~ \C[X, Y] \rightarrow \Spec~ \C[X, XY]$. Alors la valeur absolue $\eta$ sur $\C[X, XY]$ telle que $\eta (X) = \frac{1}{2}, \eta(XY) = 1$ ne se relève pas en une valeur absolue bornée par 1 sur $\C[X,Y]$ mais a des relevés sur $\C[X,Y]$.
\end{rem}

Dans un premier temps, on supposera que le morphisme $f$ est plat et surjectif. La platitude n'est pas nécessaire mais la preuve présente déjà les arguments nécessaires pour montrer la surjectivité de $f^\beth$.  

\begin{prop}\label{si le morphisme est plat, surjectivité des beth}
Soit $f : X \rightarrow Y$ morphisme plat de type fini surjectif où $X, Y$ sont des schémas intègres, de type fini sur $k$. Alors l'application induite $X^\beth \rightarrow Y^\beth$ est surjective. 
\end{prop}

\begin{proof}
Puisque $X^\beth$ est compact et $f^\beth :  X^\beth \rightarrow Y^\beth$ est continue, pour montrer la surjectivité, il suffit d'atteindre un sous-ensemble dense de $Y^\beth$. Par la proposition \ref{Les valuations divisorielles sont denses}, il suffit de montrer que l'on atteint toutes les valuations divisorielles. Ces valuations proviennent de la situation géométrique suivante : on prend un sous-schéma fermé que l'on éclate puis l'on normalise l'éclatement pour obtenir une valuation.

Soit donc $D$ un sous-schéma fermé de $Y$ et notons $Y_D$ l'éclatement de $D$ dans $Y$, soit $D'$ l'image réciproque par $f$ de $D$ comme $f$ est surjectif, $D'$ n'est pas tout $X$ et notons donc $X_{D'}$ l'éclatement de $D'$ dans $X$.

Par le lemme 31.32.3 de \cite{stacks-project}, on a un diagramme cartésien :
\begin{tikzcd}
X_{D'}\arrow[r, "\pi_{X_{D'}}"] \arrow[d, "f'"] & X\arrow[d, "f"]\\
Y_D\arrow[r, "\pi_{Y_D}"]&Y
\end{tikzcd}.

Notons $n : \tilde{X}_{D'} \rightarrow X_{D'}$ la normalisation de $X$, alors $(\pi_{Y_D} \circ f' \circ n)^{-1}(D)$ est un diviseur (de Weyl) de $\tilde{X}_{D'}$ Ainsi, $\mathcal{O}_{\tilde{X}_{D'}, D'}$ est un anneau de valuation qui relève la valuation divisorielle sur $Y$ et qui se factorise par une valuation centrée sur $X$ par le diagramme ci-dessus. Ainsi, toutes les valuations divisorielles sur $Y$ se relèvent en une valuation centrée sur $X$.

Ainsi, $X^\beth \rightarrow Y^\beth$ est surjectif.
\end{proof}

On a le même résultat avec $f$ seulement surjectif et non nécessairement plat. 

\begin{prop}\label{si le morphisme est surjectif, surjectivité des beth}
Soit $f : X \rightarrow Y$ morphisme de type fini surjectif où $X, Y$ sont des schémas intègres, de type fini sur $k$. Alors l'application induite $X^\beth \rightarrow Y^\beth$ est surjective. 
\end{prop}

\begin{proof}

Comme pour la preuve précédente, on se ramène aux valuations divisorielles.

Puisque $X^\beth$ est compact et $f^\beth :  X^\beth \rightarrow Y^\beth$ est continue, pour montrer la surjectivité, il suffit d'atteindre un sous-ensemble dense de $Y^\beth$. Par la proposition \ref{Les valuations divisorielles sont denses}, il suffit de montrer que l'on atteint toutes les valuations divisorielles. Ces valuations proviennent de la situation géométrique suivante : on prend un sous-schéma fermé que l'on éclate puis l'on normalise l'éclatement pour obtenir une valuation. 

Soit donc $D$ un sous-schéma fermé de $Y$ et notons $Y_D$ l'éclatement de $D$ dans $Y$, on va considérer le diagramme cartésien suivant :
\begin{tikzcd}
E\arrow[r, "\pi_{X}"] \arrow[d, "f'"] & X\arrow[d, "f"]\\
Y_D\arrow[r, "\pi_{Y_D}"]&Y
\end{tikzcd}.

Comme $f$ et $\pi_{Y_D}$ sont surjectives, c'est aussi le cas de $f'$. Soit $\eta_D$ le point générique d'une composante irréductible du diviseur exceptionnel de $Y_D$. Alors, prenons $x \in (f')^{-1}(\eta_D)$ et considérons l'adhérence de $\overline{\lbrace x\rbrace}$ de $\lbrace x \rbrace$ dans $E$. Si c'est un sous-schéma fermé de codimension 1, on conclut comme à la proposition précédente, sinon on considère l'éclaté $E_x$ de $\overline{\lbrace x\rbrace}$ dans $E$. On note $\pi_x : E_x \rightarrow E$ le morphisme. Notons $\eta_{\tilde{x}}$ le point générique d'une composante irréductible du diviseur exceptionnel de $E_x$, alors $(f'\circ \pi_x) (\eta_{\tilde{x}}) = \eta_D$. Donc, en normalisant $E_x$, on obtient un anneau de valuation $\mathcal{O}_{E_x, \eta_{\tilde{x}}}$ qui induit une valuation centrée en $X$ qui prolonge bien la valuation centrée en $Y$ de départ. 

Donc, $f^\beth : X^\beth \rightarrow Y^\beth$ est bien surjective. 
\end{proof}

\section{Existence de suite dans des fibres données}\label{SectionConstructionSuite}

Le but de cette section est de montrer que certains résultats de continuité de l'action se prolongent à la compactification hybride et permettent de lier l'action au bord à celle sur $X^{an}$. Dans un deuxième temps, le but est de construire explicitement des suites convergentes vers des points rigides en restant dans certaines fibres même dans le cas où $k$ est non dénombrable et où $X^\urcorner$ n'est donc pas métrisable. 

\subsection{Continuité de l'action au bord}

Dans toute la suite, on prendra $G$ un groupe algébrique et $X$ un schéma de type fini sur $k$ un corps non-trivialement valué et l'on notera $\hat{k}$ sa complétion. Le but de cette partie est d'étudier la continuité de l'action au bord de la compactification. On veut en particulier, étudier le comportement de suites $g_n \cdot x_n \in X^\urcorner$ dans le cas où $(g_n, x_n)$ converge dans $(G\times X)^\urcorner$.

\begin{nota}
On se ramènera souvent au cas où $X$ est affine et $G$ aussi.

Dans ce cas, on notera $R = \mathcal{O}(X)$ et $R_G = \mathcal{O}(G)$. Si $k$ est algébriquement clos, c'est un anneau de la forme $R_G = k[\frac{T_{11}}{det}, \cdots, \frac{T_{mm}}{det}]/(P_1, \cdots P_l)$ où les $P_k$ sont des polynômes en les $\frac{T_{ij}}{det}$. 

De plus, on notera $X^{an}$ si l'on analytifie $X$ suivant la valeur absolue de $\hat{k}$ et $X^{hyb}$ si on l'analytifie suivant la valeur absolue hybride.
\end{nota}

\begin{rem}
On dispose de deux façons de voir les points de l'analytifié d'un schéma. Soit $X$ un schéma affine, on suppose ici que $X = \ana^n_k$.

Soit $x$ un point de $X^{an}$ et notons $\eta_x$ la semi-norme correspondante.

On sait que $\eta_x$ est une semi-norme sur $R = k[T_1, \cdots, T_d]$ et on note $\mathcal{H}(x)$ son corps résiduel, donc le complété de $R/\ker~\eta_x$ muni de la norme $\eta_x$. On obtient un point de $\mathcal{H}(x)^d$ en regardant l'image de chaque $T_i$ dans $\mathcal{H}(x)$. Ainsi, on associe à chaque semi-norme un point de $\mathcal{H}(x)^d$.

Réciproquement si $x = (x_1, \cdots, x_d)$ est un point de $K^d$ où $(K, |\cdot|_K)$ est une extension valuée de $k$, on peut définir une semi-norme $\eta_x$ associée à $x$, de la façon suivante : 

Soit $P \in R$, alors $\eta_x(P) = |P(x_1, \cdots, x_d)|_K$. Cela signifie que l'on évalue en les coefficients de $x$. Il faut néanmoins faire attention, si $x$ est à coefficients dans un corps $K$, le corps résiduel de $x$ peut être un corps très différent de $K$, c'est la complétion d'un sous-corps de $K$.

On peut ainsi voir toute semi-norme $\eta_x$ comme étant l'évaluation sur le corps $\mathcal{H}(x)$ en les coefficients de $x$.

Dans toute la suite, on identifiera donc les semi-normes avec les points de $\ana^{n,an}_{\h(x)}(\overline{\h(x)})$. Pour savoir quel point de vue l'on adopte, on notera $\eta_x$ la semi-norme associée au point $x$.
\end{rem}

\begin{rem}
Les résultats sont donnés sur les points $X^{an}$, mais tous les résultats tiennent si les $x_n$ ne sont pas des points de $X^{an}$ mais simplement des points de $X_\infty$. Cela vient du fait que l'on voit toutes les semi-normes comme des évaluations et on se ramène donc au même problème que sur les points de $X^{an}$. Si des différences apparaissent dans les preuves, on le notera en remarque.
\end{rem}

Remarquons que l'on dispose de sections continues, tout d'abord une section continue sur un ouvert contenant tout le bord.

\begin{prop}\label{prop section continue bord}
Notons $X = \ana^n_\C$, alors on peut analytifier $X$ selon la valeur absolue hybride sur $\C$. On note $i$ l'immersion ouverte définie dans \cite{PoineauCompactificationHybride} à la définition 4.5 et redéfinie à la définition \ref{immersion de Xan dans compactification} qui plonge $X(\C)$ dans $X^\urcorner$.

Prenons alors $r >1$ et l'ouvert $\mathcal{U}_r := X^\urcorner \setminus i(\overline{B}(0,r))$ où $\overline{B}(0,r)$ désigne la boule fermée de $\C^n$ munie de la valeur absolue usuelle, de centre 0 et de rayon $r$. C'est un ouvert contenant le bord de $X^\urcorner$. 

Soit $x \in \mathcal{U}_r$, alors définissons $\eta_x \in X^+$ comme étant l'unique point vérifiant $\pi(\eta_x) = x$ et $\max ~\eta_x(T_i) = r$ où $\pi : X^+ \rightarrow X^\urcorner$ est la projection.

Alors l'application
\begin{equation*}
\Phi_r : \left\{
    \begin{array}{ll}
        \mathcal{U}_r \rightarrow X^+, \\
        x \mapsto \eta_x
    \end{array}
\right.
\end{equation*} 
est continue. De plus, $\pi \circ \Phi_r = Id_{\mathcal{U}_r}.$ 
\end{prop}
\begin{rem}\label{section continue bord corps non-archi}
On a pris $\C$, mais on peut faire le même raisonnement avec $k$ un corps non-archimédien ou $k$ un corps archimédien complet ou non tel que $\Q, \R$ par exemple.
\end{rem}
\begin{proof}
Posons $F_r = \lbrace \eta \in X^+ | \max~\eta(T_i) = r \rbrace$ de telle sorte que $\Phi_r(\mathcal{U}_r) \subset F_r$ et notons $\pr$ la projection de $X^{hyb}$ vers $M(\C, |\cdot|_{hyb}) = [0,1]$.

Soit $V$ un ouvert de $X^+$. Définissons alors $A := \lbrace \eta\in X^+ | \pr(\eta)\neq 0, \max~\eta^{\frac{1}{\pr(\eta)}}(T_i) \leq r\rbrace$.

Alors, $\Phi_r^{-1}(V) = \pi((V \cap F_r) \setminus A)$ donc $\Phi_r^{-1}(V)$ est ouvert ssi $T((V \cap F_r) \setminus A)$ l'est. On rappelle que $T((V \cap F_r) \setminus A)$ est défini à la définition \ref{def trajectoire ensemble}. 

Il suffit de montrer que $\Phi_r^{-1}(V)$ est ouvert pour $V = \lbrace \eta \in X^+ | s_1 < \eta(P) < s_2 \rbrace$, où $s_1 < s_2, P \in k[T_1, \cdots, T_n]$, comme ces ensembles engendrent la topologie. Dans la suite, on considéra donc que $V$ est de cette forme.

Soit $x \in T((V\cap F_r)\setminus A)$. Soient $1 < a_1 < \max~x(T_i) < a_2, b_1 < x(P) < b_2$ que l'on choisira plus tard. Posons $F_{a_1, a_2} := \lbrace \eta \in X^+| a_1 < \max~\eta(T_i) < a_2 \rbrace \cap \pr^{-1}([0, \frac{\ln a_1}{\ln r}[)$ et $V_{b_1, b_2} = \lbrace \eta \in X^+ | b_1 < \eta(P) < b_2 \rbrace$. Montrons que l'on peut choisir $a_1, a_2, b_1, b_2$ pour que $F_{a_1, a_2} \cap V_{b_1, b_2}$ soit un voisinage ouvert de $x$ dans $T((V \cap F_r)\setminus A)$.

Trouvons des conditions pour que $V_{b_1, b_2} \cap F_{a_1, a_2} \subset T(V \cap F_r)$. Soit $y \in V_{b_1, b_2} \cap F_{a_1, a_2}$.

Comme $y \in F_{a_1, a_2}$, on a $a_1 < \max~y(T_i) <a_2$. Soit maintenant $\alpha$ tel que $\max~y^\alpha(T_i) = r$, alors $\frac{\ln r}{\ln a_2} < \alpha < \frac{\ln r}{\ln a_1}$. Comme $y \in \pr^{-1}([0, \frac{\ln a_1}{\ln r}[), I_y \supset [0, \frac{\ln r}{\ln a_1}[$ et donc $y^\alpha$ est bien défini.

Le but est d'avoir $y^\alpha \in V$ et $y^\alpha \notin A$ pour que $y \in T((V\cap F_r)\setminus A)$. 

Comme $\pr(y) < \frac{\ln a_1}{\ln r}$, on a $\frac{1}{pr(y)} > \frac{\ln r}{\ln a_1} > \alpha$ et donc $\max~y^{\frac{1}{pr(y)}}(T_i) > \max~y^\alpha(T_i) = r$ et donc $y \notin A$, ce qui est équivalent au fait que $y^\alpha$ ne soit pas dans $A$ comme $\pr(y^\alpha) = \alpha ~ \pr(y)$. Il reste à montrer que $y^\alpha \in V$.

Or on a $b_1^\alpha < y^\alpha(P) <b_2^\alpha$.  Donc on veut 
\begin{equation*}
b_1^\alpha > s_1, b_2^\alpha < s_2
\end{equation*}
et comme $b_1^\alpha > \min (b_1^\frac{\ln r}{\ln a_2}, b_1^\frac{\ln r}{\ln a_1})$ et $b_2^\alpha < \max (b_2^\frac{\ln r}{\ln a_2}, b_2^\frac{\ln r}{\ln a_1})$, on veut :
\begin{align*}
\left\{
    \begin{array}{ll}
        \min (b_1^\frac{\ln r}{\ln a_2}, b_1^\frac{\ln r}{\ln a_1}) &> s_1 \\
        \max (b_2^\frac{\ln r}{\ln a_2}, b_2^\frac{\ln r}{\ln a_1}) &< s_2 
    \end{array}
\right.\\
\iff
\left\{
    \begin{array}{ll}
        \max (s_1^\frac{\ln a_2}{\ln r}, s_1^\frac{\ln a_1}{\ln r}) &< b_1 \\
        \min (s_2^\frac{\ln a_2}{\ln r}, s_2^\frac{\ln a_1}{\ln r}) &> b_2 
    \end{array}
\right.
\end{align*} 
Comme, l'on veut $b_1 < b_2$, il suffit d'avoir 
\begin{align*}
s_1^{\ln a_i} &< s_2^{\ln a_j}\\
s_1 &< s_2 ^{\frac{\ln a_j}{\ln a_i}},
\end{align*}
pour $i,j \in \lbrace 1,2 \rbrace$. Et cela est possible si $a_1$ et $a_2$ sont proches de la valeur $\max~x(T_i)$. Donc il existe $1 < a_1 < a_2, b_1 < b_2$ tel que $V_{b_1, b_2} \cap F_{a_1, a_2} \subset T(V\cap F_r)$ il faut maintenant montrer que l'on peut aussi les choisir de manière à ce que ce soit un voisinage de $x$. 

Soit $\beta$ tel que $x^\beta \in (V\cap F_r)\setminus A$ donc $x \notin A$ (ce qui est équivalent à $x^\beta \notin A$). Donc $\max~x^{\frac{1}{\pr(x)}} (T_i)>r$ ce qui implique que $\pr(x) < \frac{\ln \max~x(T_i)}{\ln r}$ et donc si on prends $a_1$ suffisamment proche de $\max~x(T_i)$, on aura $\pr(x) < \frac{\ln a_1}{\ln r}$ et donc $x \in F_{a_1,a_2}$.

On sait de plus que $x^\beta \in V$, donc 
\begin{align*}
s_1 &< x^\beta(P) <s_2 \\
\exists \epsilon > 0 , s_1 + \epsilon &\leq x^\beta(P) \leq s_2 -\epsilon\\
 \exists \epsilon > 0 , (s_1 + \epsilon)^{\frac{1}{\beta}} &\leq x(P) \leq (s_2 -\epsilon))^{\frac{1}{\beta}}
\end{align*}
Donc si on choisit $b_1, b_2$ tel que 
\begin{equation*}
\left\{
	\begin{array}{ll}
		b_1 < \min ((s_1 + \epsilon)^{\frac{\ln a_1}{\ln r}}, (s_1 + \epsilon)^{\frac{\ln a_2}{\ln r}}) \\
		b_2 > \max ((s_2 - \epsilon)^{\frac{\ln a_1}{\ln r}}, (s_2 - \epsilon)^{\frac{\ln a_2}{\ln r}})
	\end{array}
\right.
\end{equation*}
on aura bien $x\in V_{b_1, b_2}$. Mais comme
\begin{equation*}
\left\{
    \begin{array}{ll}
        \max (s_1^\frac{\ln a_2}{\ln r}, s_1^\frac{\ln a_1}{\ln r}) &< b_1 \\
        \min (s_2^\frac{\ln a_2}{\ln r}, s_2^\frac{\ln a_1}{\ln r}) &> b_2 
    \end{array}
\right.
\end{equation*}
il faut avoir 
\begin{equation*}
\left\{
    \begin{array}{ll}
        s_1^{\ln a_i} &< (s_1 +\epsilon)^{\ln a_j} \\
        s_2^{\ln a_i} &> (s_2 -\epsilon)^{\ln a_j} 
    \end{array}
\right.
\end{equation*}
pour $i, j \in \lbrace 1, 2 \rbrace$, ce qui est possible pour $a_1, a_2$ proche. Donc, en choisissant $a_1, a_2$ suffisamment proche l'un de l'autre et $a_1$ proche de $\max~x(T_i)$, on peut choisir $b_1, b_2$ tel que $V_{b_1, b_2} \cap F_{a_1, a_2}$ soit un voisinage de $x$ dans $T((V \cap F_r)\setminus A)$.
\end{proof}

Dès lors on a le corollaire suivant, nous permettant de relever des suites à $X^+$.

\begin{cor}\label{existence de releve pour les suites convergentes}
Notons $X = \ana^n_k$ où $k$ est un corps valué muni d'une valeur absolue non-triviale.

Soient $x_n \in X^\urcorner \rightarrow x \in X^\urcorner$. Si $y \in X^+$ est tel que $\pi(y) = x$, alors il existe $y_n \in X^+$ tel que $\pi(y_n) = x_n$ et tel que $y_n \rightarrow y$. 
\end{cor}

\begin{proof}
On distingue 2 cas :
\begin{itemize}
\item Tout d'abord le cas où $x \in \delta X$. Alors on est dans le cas de la proposition \ref{prop section continue bord}, et $x$ est dans tout les $\mathcal{U}_r$ pour $r >1$. Soit $r>1$ tel que $\Phi_r(x) = y$.

Comme $x_n \rightarrow x$, à partir d'un certain rang, on a $x_n \in \mathcal{U}_r$. Ainsi, par continuité de $\Phi_r$, les $y_n := \Phi_r(x_n)$ conviennent. 

\item Ensuite le cas où $x\notin \delta X$, alors on est dans le cas de la proposition \ref{prop section continue partie archi}, et comme $x_n \rightarrow x$, alors à partir d'un certain rang, tous les $x_n$ et $x$ sont dans l'image de l'homémorphisme et donc les $i^{-1}(x_n)$ conviennent.
\end{itemize}
\end{proof}

On va maintenant s'intéresser à l'action de $G$ sur $X$. 

\begin{rem}
Redonnons quelques propriétés de l'analytification d'un produit fibré dans le cas où tous les schémas sont affines.

Définissons $A := R_G \otimes_k R$ que l'on utilisera dans toute la suite.

Alors, on a : 
\begin{align*}
&(G \times X)^{hyb} = \Spec(A)^{hyb}.
\end{align*}

On prend ici l'analytification avec $k$ muni de la valeur absolue hybride.
 
Soit $x \in X^{hyb}$ alors si on note $pr_2$ la projection sur le deuxième facteur, on a :

\begin{align*}
\pr_2^{-1}(x) &= \lbrace \eta : A \rightarrow \R_+, \eta|R = \eta_x\rbrace \\
&= G_{\mathcal{H}(x)}^{an}
\end{align*}

où l'analytification est ici prise avec $\mathcal{H}(x)$ muni de sa valeur absolue induite. On utilisera la notation $G^{an}_{\mathcal{H}(x)}$ dans toute la suite.
\end{rem}








On va maintenant étudier des convergences de suites dans des compactifications hybride. On utilisera toujours le corollaire \ref{existence de releve pour les suites convergentes}, qui nous permettra de choisir un relevé de notre limite ainsi que des relevés des points de la suite qui convergent vers le relevé de la limite. 

\begin{nota}\label{On prend des releves de suite}
Dans toute la suite, lorsque l'on notera $x_n \rightarrow x$, on notera toujours $\eta_x$ pour un relevé de $x \in X^{hyb}$ et $\eta_{x_n} \in X^{hyb}$ un relevé de $x_n$. Dans le cas où $x_n \notin \delta X$, on prendra $\eta_{x_n}$ comme étant l'image réciproque de $x_n$ via l'immersion $i : (X \otimes_k \hat{k})^{an} \rightarrow X^\urcorner$. Alors par le corollaire \ref{existence de releve pour les suites convergentes}, on sait qu'il existe des $y_n \in X^{hyb}$ avec $y_n \rightarrow \eta_x$ tel que $y_n$ et $\eta_{x_n}$ soient reliés par le flot. On notera alors $\epsilon \geq 0$ tel que $y_n = \eta_{x_n}^{\epsilon_n}$. Dans le cas, où $x_n \notin \delta X$, alors $\epsilon_n \in [0,1]$ et $\epsilon_n \rightarrow 0$ si $x \in \delta X$.
Ainsi, dans le cas où $X$ est affine, on a :

Pour tout polynôme $P \in \mathcal{O}(X), |P(coeff~de~x_n)|^{\epsilon_n} \rightarrow \eta_x(P)$.
\end{nota}

Tout d'abord étudions l'effet de l'action de $G$ sur les limites séquentielles.

\begin{prop}\label{g_n x_n tend vers g x}
Soit $X$ une $k$-variété sur un corps non-trivialement valué et soit $G$ un groupe agissant sur $X$.

Soit $(g_n, x_n) \in (G \times X)^{an}$, on les voit comme des éléments de $(G\times X)^\urcorner$. Supposons qu'il existe $(g, x)\in \delta (G \times X)$ où on voit $g \in G_{\mathcal{H}(x)}^{an}$ tel que  $(g_n, x_n) \rightarrow (g, x)$. Notons $\eta_x \in X_\infty$ un relevé de $x$. On note de même $\eta_{(g,x)} \in (G\times X)_\infty$ le relevé de $(g,x)$ tel que $pr^1(\eta_{(g,x)}) = \eta_x$ où $pr^1$ est la projection sur la première coordonnée. On note alors $\eta_{g\cdot x}$ le relevé de $g\cdot x$ tel que $\Phi^{hyb}(\eta_{(g,x)}) = \eta_{g\cdot x}$ où $\Phi^{hyb}$ est l'analytifié de $\Phi : G\times X \rightarrow X$.

Alors il existe $\epsilon_n \rightarrow 0$ tel que pour $(\eta_{g_n \cdot x_n},{\epsilon_n}) \rightarrow \eta_{g \cdot x}$ où $\eta_{g \cdot x} \in X^\infty \cup X^\beth$ et $(\eta_{g_n \cdot x_n},{\epsilon_n})$ signifie que l'on regarde chaque $\eta_{g_n \cdot x_n}$ dans la fibre $\pr^{-1}(\epsilon_n)$ où $pr : X \rightarrow M(k_{hyb})$.
\end{prop}

\begin{rem}
Pour le cas où $(g_n, x_n) \in \delta (G \times X)$, on a le même résultat, mais il n'y a pas la condition sur $\epsilon_n \rightarrow 0$.

Dans le cas où $g\cdot x \in X_\infty$, cela signifie exactement que $g_n\cdot x_n \rightarrow g \cdot x \in X^\urcorner$ mais si $g\cdot x \in X^\beth$, on ne peut pas parler de convergence de $g_n\cdot x_n$ vers $g\cdot x$ sans parler des $\epsilon_n$ et sans parler de $X^{hyb}$.
\end{rem}

\begin{ex}
On peut avoir les deux cas : $\eta_{g\cdot x} \in X^\beth$ ou $\eta_{g\cdot x} \in X^\infty$. 

Prenons par exemple le cas où $X = Rat_1, G = SL_2$.

Soit $t \in \C$, avec $|t| = r < 1$ alors posons $f_n = \frac{z - t^n}{t^{-n}z +1} \in Rat_1(\C)$. Alors en prenant $\epsilon_n = \frac{1}{n}$, on obtient la convergence des $f_n$ vers $f = \frac{z - T}{T^{-1}z +1}\in Rat_1^\infty$ dont le corps résiduel est $\C((T))$. 

Prenons alors $M_n = \begin{pmatrix}
t^{\frac{n}{2}} & 0\\
0 & t^{-\frac{n}{2}}\\
\end{pmatrix}$. Ainsi, $(M_n, f_n) \in Rat_1(\C) \rightarrow (M,f)$ où $M \in SL_2^{an}(\mathcal{H}(f))$ et le corps résiduel de la semi-norme associée à $(M, f)$ est $\C((T^{\frac{1}{2}}))$ et $M$ peut être vue comme un élément de $SL_2(\C((T^{\frac{1}{2}})))$ où $ M = \begin{pmatrix}
T^{\frac{1}{2}} & 0\\
0 & T^{-\frac{1}{2}}\\
\end{pmatrix}$. 

De plus, $f_n^{M_n} = \frac{z -1}{z+1}$ et donc si l'on note $g = \frac{z -1}{z+1}$, on peut voir $g$ comme un élément de $Rat_1^\beth$ dont le corps résiduel est $\C$ muni de la valeur absolue triviale. Ainsi, $f_n^{M_n} \rightarrow g = f^M \in Rat_1^\beth$.

Et si, on avait pris $M_n = Id$ pour tout $n$, alors $f_n^{M_n}$ aurait convergé vers $f$ qui est dans le bord de $Rat_1$.
\end{ex}

\begin{proof}
Notons comme dans la notation \ref{On prend des releves de suite}, $\eta_{(g_n, x_n)}$ les relevés de $(g_n, x_n) \in (G\times X)^\urcorner$. Alors il existe $\epsilon_n \rightarrow 0$ tel que $\eta_{(g_n, x_n)}^{\epsilon_n} \rightarrow \eta_{(g,x)} \in (G\times X)^{hyb}$. 

Alors par continuité de $\Phi^{hyb}$, $(\eta_{g_n \cdot x_n},{\epsilon_n}) = \Phi^{hyb}(\eta_{(g_n, x_n)},{\epsilon_n}) \rightarrow \Phi^{hyb}(\eta_{(g,x)}) = \eta_{g\cdot x}$.
\end{proof}

\begin{prop}\label{va de fg_n x_n est va de g x}
Soit $(g_n, x_n) \in (G \times X)^{an}$ et soit $(g,x) \in \delta(G \times X)$ une valeur d'adhérence de la suite vue dans $(G\times X)^\urcorner$ où $g \in G_{\mathcal{H}(x)}^{an}$. Alors $g\cdot x$ est une valeur d'adhérence de $g_n\cdot x_n$ dans $X^{hyb}$.
\end{prop}

\begin{proof}
Comme $(g,x)$ est une valeur d'adhérence de $(g_n, x_n)$, il existe une suite généralisée à valeurs dans $\lbrace (g_n, x_n, \epsilon_n) \rbrace$ qui converge vers $(g,x)$. On peut donc appliquer la proposition \ref{g_n x_n tend vers g x} à cette suite généralisée (la démonstration est la même dans ce cas là). Donc $g\cdot x$ comme limite d'une suite généralisée à valeurs dans $\lbrace (g_n\cdot x_n, \epsilon_n)\rbrace$ est une valeur d'adhérence de $(g_n\cdot x_n, \epsilon_n)$ dans $X^{hyb}$.
\end{proof}

\subsection{Existence de suite}

On va maintenant faire le chemin inverse : on va prendre une suite $x_n \rightarrow x' \in X^\urcorner$, on note $x \in X^{hyb}$ un relevé de $x'$. Quitte à supprimer les premiers termes de la suite, on peut supposer que $X$ est affine sans perte de généralité. On suppose même que $X = \ana^d_k$. En général $X$ est seulement un fermé $V(I)$ de l'espace affine. Mais via la projection $k[T_1, \cdots, T_d] \rightarrow k[T_1, \cdots, T_d]/I$, on peut donner une valeur à $|P(x)|$ pour tout $P \in k[T_1, \cdots, T_d]$ et donc on peut se ramener au cas où $X = \ana^d_k$.

Le but est de prendre $g \in G^{an}_{\h(x)}(\overline{\mathcal{H}(x)})$ et construire des éléments $g_n \in G^{an}_{\h(g_n)}$ telles que $(g_n, x_n) \rightarrow (g,x')$. Pour cela, on va montrer que cela est vrai pour tout schéma $Y$sur $k$ de type fini, en commençant par le cas où $Y$ est l'espace affine de dimension $m$. On écrit les propositions dans le cas où $x_n \in X(k) \rightarrow x'\in \delta X$ mais les résultats restent vraies si $x \in X^{an}$ et non nécessairement $\delta X$ ou si les $x_n$ sont des éléments de $X^{an}$ ou des éléments de $\delta X$. On écrira les changements s'il y en a en remarque.

\begin{defi}\label{def beta algebrique}
Soit $x_n \in X(k) \rightarrow x' \in \delta X, ~\beta$ algébrique sur $\mathcal{H}(x)$ où $x \in X^{hyb}$ désigne un relevé de $x'\in \delta X$. Alors notons
 $(\beta, x)$ le point de $(\ana^1 \times X)^{an}_0$ correspondant à la semi-norme sur $k[T, T_1, \cdots, T_d]$, où l'on évalue chaque polynôme dans $\mathcal{H}(x)(\beta)$ en 
  $\beta$ et les coefficients de $x$.
\end{defi}

\begin{prop}\label{existence de beta_n qui tends vers beta}
Soit $x_n \in X(k) \rightarrow x' \in \delta X, ~\beta$ algébrique sur $\mathcal{H}(x)$ dont le polynôme minimal sur $\mathcal{H}(x)$ est à coefficient dans $\Frac(R/\ker~\eta_x)$ où $x$ est un relevé de $x'$. Alors il existe $\beta_n \in \overline{k}$ tel que $(\beta_n, x_n) \in (\ana^1 \times X) (\overline{k}) \rightarrow (\beta, x') \in \delta (\ana^1 \times X)$.
\end{prop}

\begin{rem}
Dans le cas où $x_n \in \delta X$ ou $x_n \in X^{an}\backslash X(k)$ alors $\beta_n$ sera un élément de la clôture algébrique de $\mathcal{H}(x_n)$.
\end{rem}
\begin{proof}
Notons $\mu_\beta$ le polynôme minimal de $\beta$, en ne le prenant pas unitaire, on peut supposer que $\mu_\beta$ est à coefficient dans $R/\ker~\eta_x$ donc ses coefficients sont des polynômes en les coefficients de $x$, donc il existe $P_k \in R, 0 \leq k \leq l$ tel que $\mu_\beta = \sum P_k(x_1, \cdots, x_d)T^k$. Prenons alors $\mu_n = \sum P_k(x_{1,n}, \cdots, x_{d,n})T^k \in k[T]$ le polynôme obtenu en prenant les coefficients de $x_n$. Et prenons $\beta_n$ une racine de $\mu_n \in \overline{k}$.

Comme $(\beta_n, x_n)$ est à valeurs dans un compact ($(\ana^1 \times X)^\urcorner$), cette suite admet une valeur d'adhérence que l'on note $(a,y)$. $y$ est un élément de $X^{an}_0/\Phi$  où $\Phi$ désigne le flot et $a$ correspond à un point de $\ana^1_{\mathcal{H}(y)}$ et on le voit comme un élément de son corps résiduel.

Comme $(a,y) \in (\ana^1 \times X)^\urcorner$, on peut le relever en un point $(\tilde{a}, \tilde{y}) \in (\ana^1 \times X)^{hyb}$ alors si $\tilde{y} \notin X^\beth$ nécessairement, $y = x \in X^\urcorner$. 

Montrons donc que $\tilde{y}$ ne peut pas être un élément de $X^\beth$. Sinon, on sait qu'il existe $\alpha_n \in [0,1]$ tel que $\eta_{(\beta_n, x_n)}^{\alpha_n}$ soit une valeur d'adhérence de $(\tilde{a}, \tilde{y}) \in (\ana^1 \times X)^{hyb}$. 

Comme $\beta_n$ est une racine de $\mu_n$, on a nécessairement :
\begin{equation*}
|\beta_n| \leq max(1, | \frac{P_0(x_{1,n}, \cdots, x_{d,n})}{P_l(x_{1,n}, \cdots, x_{d,n})}|, \cdots,| \frac{P_{l-1}(x_{1,n}, \cdots, x_{d,n})}{P_l(x_{1,n}, \cdots, x_{d,n})}|).
\end{equation*} 

D'où,
\begin{equation*}
|\beta_n|^{\alpha_n} \leq max(1, | \frac{P_0(x_{1,n}, \cdots, x_{d,n})}{P_l(x_{1,n}, \cdots, x_{d,n})}|^{\alpha_n}, \cdots, | \frac{P_{l-1}(x_{1,n}, \cdots, x_{d,n})}{P_l(x_{1,n}, \cdots, x_{d,n})}|^{\alpha_n}) \rightarrow 1.
\end{equation*} 
Le terme de droite tends vers $1$ car $\tilde{y}$ est dans $X^\beth$ et donc quitte à extraire $\frac{\alpha_n}{\epsilon_n} = 0$.

Or, $a$ est une valeur d'adhérence de $\eta_{\beta_n}^{\alpha_n}$ et donc $a \in \ana^{1, \beth}$ et donc $(a,y) \notin (\ana^1 \times Rat_d)^\urcorner$, ce qui est absurde. Donc $(a,y)$ est bien de la forme $(a,x)$.

On sait qu'il existe $P \in R[T]$ tel que $P$ s'annule en tous les $(\beta_n, x_n)$. Comme $(a,x)$ est une valeur d'adhérence de $(\beta_n, x_n)$, nécessairement $\eta_{(a,x)}(P)$ est une valeur d'adhérence de $\eta_{(\beta_n,x_n)}^{\epsilon_n}(P) = 0$. Donc, $\eta_{(a,x)}(P) = 0$. 

Donc, nécessairement $x$ est une racine de $\mu_\beta$. Donc, $x$ est un conjugué de Galois de $\beta$, mais tous les conjugués de Galois de $\beta$ définissent le même point de $\ana^1_{\mathcal{H}(x)}$ et donc $(a, x) =(\beta, x)$ et donc la suite $(\beta_n, x_n)$ n'admet qu'une valeur d'adhérence et est donc convergente vers $(\beta, x)$.
\end{proof}
\begin{prop}\label{existence suite de points qui converge vers y dans ext alg de Frac}
Soit $x_n \in X(k) \rightarrow x' \in \delta X$ et soit $y \in \ana^{m, an}_{\mathcal{H}(x)}(\Frac(R/\ker~\eta_x)(\beta))$ tel que $\beta$ soit algébrique dans $\mathcal{H}(x)$ dont le polynôme minimal est à coefficient dans $\Frac(R/\ker~\eta_x).$ Alors il existe $y_n \in \ana^{m,an}_k(\overline{k})$ tel que $(y_n, x_n) \rightarrow (y,x') \in \delta (\ana^m_k \times X)$.
\end{prop}

\begin{rem}
Si $x_n \in \delta X$ ou $x_n \in X^{an}\backslash X(k)$, alors $y_n \in \ana^{m,an}_{ \mathcal{H}(x_n)}(\Frac(R/\ker~\eta_{x_n})(\beta_n))$ où les $\beta_n$ sont ceux de la proposition \ref{existence de beta_n qui tends vers beta}. 
\end{rem}
\begin{proof}
Notons $x = (x_1, \cdots, x_d), x_n = (x_{1,n}, \cdots, x_{d,n})$ et notons $y = (y_1, \cdots, y_m)$ où chaque $y_i$ est un élément de $\Frac(R/\ker~\eta_x)(\beta)^m$.

Alors, $y$ peut-être vu comme la semi-norme de $\ana^{m,an}_{\mathcal{H}(x)}$ en l'évaluation de ces coefficients dont le corps résiduel est $\mathcal{H}(x)(\beta)$ et notons $\tilde{\eta}_x$ sa norme (qui étend $\eta_x$).

Il existe $(P_{i}, Q_{i})\in R[T]\times R$ tel que 
\begin{equation*}
y_i = \frac{P_i(x_1, \cdots, x_d, \beta)}{Q_i(x_1, \cdots, x_d)}.
\end{equation*}

Définissons donc $y_n \in \ana^{m, an}_k(\overline{k})$ en utilisant les mêmes relations. Tout d'abord par la proposition \ref{existence de beta_n qui tends vers beta}, il existe $\beta_n \in \overline{k}$ tel que $(\beta_n, x_n) \rightarrow (\beta, x)$. Puis, $y_{i,n} = \frac{P_i(x_{1,n}, \cdots, x_{d,n}, \beta)}{Q_i(x_{1,n}, \cdots, x_{d,n})}$.

Soit $P \in k[T_1, \cdots, T_m, T'_1, \cdots, T'_d]$, alors il existe $\tilde{P}, \tilde{Q} \in R[T]\times R$ tel que 
\begin{equation*}
P(\frac{P_1(x_1, \cdots, x_d, T)}{Q_1(x_1, \cdots, x_d)}, \cdots, \frac{P_m(x_1, \cdots, x_d, T)}{Q_m(x_1, \cdots, x_d)} , x_1, \cdots, x_d)= \frac{\tilde{P}(x_1, \cdots, x_d, T)}{\tilde{Q}(x_1, \cdots, x_d)}
\end{equation*} 
par exemple $\tilde{Q}$ est un produit des $Q_j$. Donc en évaluant en les coefficient de $x_n$ et en $\beta_n$, on trouve :
\begin{equation*}
P(y_{1,n}, \cdots, y_{m,n}, x_{1_n}, \cdots, x_{d,n})= \frac{\tilde{P}(x_{1,n}, \cdots, x_{d,n},\beta_n)}{\tilde{Q}(x_{1,n}, \cdots, x_{d,n})}.
\end{equation*} 

Alors, 
\begin{align*}
|P(y_{1,n}, \cdots, y_{m,n}, x_{1_n}, \cdots, x_{d,n})|^{\epsilon_n} &\rightarrow \eta_{(\beta, x)}(\frac{\tilde{P}}{\tilde{Q}})\\
&= \tilde{\eta}_x(\frac{\tilde{P}(x_1, \cdots, x_d, \beta)}{\tilde{Q}(x_1,\cdots, x_d)})\\
&= \tilde{\eta}_x(P(y_1, \cdots, y_m, x_1, \cdots, x_d))\\
&= \eta_{(y,x)}(P).
\end{align*}
Où pour la première ligne on utilise le fait que $(\beta_n, x_n) \rightarrow (\beta, x)$ ensuite on voit la semi-norme $\eta_{(\beta, x)}$ comme une évaluation, puis on utilise le lien entre $P$ et $\tilde{P}, \tilde{Q}$ puis enfin, on voit la semi-norme $\eta_{(y,x)}$ comme une évaluation.
\end{proof}
\begin{rem}
En particulier, si $y \in \ana^{m,an}_{\mathcal{H}(x)}(\Frac(R/\ker~\eta_x))$, il existe $y_n \in \ana^{m,an}_k(k)$ tel que $(y_n, x_n) \rightarrow (y,x)$ en prenant $\beta = 1$ dans la proposition précédente.
\end{rem}
On peut maintenant utiliser un procédé diagonal pour réussir à atteindre tout $\ana^{m,an}_{\mathcal{H}(x)}(\overline{\mathcal{H}(x)}).$
\begin{prop}\label{existence convergence vers y dans H(x)^m}
Soit $x_n \in X(k) \rightarrow x' \in \delta X$ et soit $y\in \ana^{m,an}_{\mathcal{H}(x)}(\mathcal{H}(x))$. Alors il existe $y_n \in \ana^{m,an}_k(\overline{k})$ tel que $(y_n, x_n) \rightarrow (y,x') \in \delta (\ana^m \times X)$.
\end{prop}

\begin{rem}
Dans le cas où $x_n \in \delta X$ ou $x_n \in X^{an}\backslash X(k)$, alors les $y_n$ sont des éléments de $\ana^{m,an}_{\mathcal{H}(x_n)}(\mathcal{H}(x_n))$.
\end{rem}
\begin{proof}
Notons $x = (x_1, \cdots, x_d), x_n = (x_{1,n}, \cdots, x_{d,n}), y = (y_1, \cdots, y_m) \in \ana^{m,an}_{\mathcal{H}(x)}(\mathcal{H}(x))$.

Soit $(y_{i}^k)_{k \in \N} \in \Frac(R/\ker~\eta_x),$ une approximation de $y_{i}$ tel que $\eta_x(y_{i}^k - y_{i}) \leq \frac{1}{2^{k+1}}, 1\leq i\leq m.$ 

Alors, il existe $\tilde{A}_{i}^k, \tilde{B}_{i}^k \in R/\ker~\eta_x, y_{i}^k = \frac{\tilde{A}_{i}^k}{\tilde{B}_{i}^k}$, on choisit des relèvements $A_{i}^k, B_{i}^k$ dans $R$.

Comme $\eta_x(B_{i}^k) \neq 0$ et $x_n \rightarrow x, |B^k_{i}(x_{1,n}, \cdots, x_{d,n})|^{\epsilon_n} \rightarrow \eta_x(B_{i}^k) \neq 0 $ donc $|B^k_{j}(x_{1,n}, \cdots, x_{d,n})|^{\epsilon_n} \neq 0$ pour n assez grand, on peut donc diviser par cette quantité. 

De plus comme $x_n \rightarrow x$,

\begin{align*}
\forall 1\leq i \leq m, \forall k, \exists N_{i}^k, \forall n \geq N_{i}^k, \Biggr\lvert \eta_x(y_{i}^k) - |\frac{A^k_{i}(x_{1,n}, \cdots, x_{d,n})}{B^k_{i}(x_{1,n}, \cdots, x_{d,n})}|^{\epsilon_n}\Biggr\rvert \leq \frac{1}{2^{k+1}}.
\end{align*}
on peut prendre $N_{i}^k$ minimaux parmi cette condition et tel qu'ils soient strictement croissants en $k$.

Cela nous permet également de contrôler la différence entre $|\frac{A^k_{i}(x_{1,n}, \cdots, x_{d,n})}{B^k_{i}(x_{1,n}, \cdots, x_{d,n})}|^{\epsilon_n}$ et $|\frac{A^{k+1}_{i}(x_{1,n}, \cdots, x_{d,n})}{B^{k+1}_{i}(x_{1,n}, \cdots, x_{d,n})}|^{\epsilon_n}$, en effet :
\begin{align}\label{controle diff entre A^k/B^k et A^k+1/B^k+1}
&\forall n \geq N_{i}^{k+1},\nonumber\\
&\Biggr\lvert|\frac{A^k_{i}(x_{1,n}, \cdots, x_{d,n})}{B^k_{i}(x_{1,n}, \cdots, x_{d,n})}|^{\epsilon_n} - |\frac{A^{k+1}_{i}(x_{1,n}, \cdots, x_{d,n})}{B^{k+1}_{i}(x_{1,n}, \cdots, x_{d,n})}|^{\epsilon_n} \Biggr\rvert \nonumber\\
&\leq \Biggr\lvert |\frac{A^k_{i}(x_{1,n}, \cdots, x_{d,n})}{B^k_{i}(x_{1,n}, \cdots, x_{d,n})}|^{\epsilon_n} - \eta_x(y_{i}^k)\Biggr\rvert + \Biggr\lvert \eta_x(y_{i}^k) - \eta_x(y_{i}^{k+1}) \Biggr\rvert +  \Biggr\lvert\eta_x(y_{i}^{k+1}) - |\frac{A^{k+1}_{i}(x_{1,n}, \cdots, x_{d,n})}{B^{k+1}_{i}(x_{1,n}, \cdots, x_{d,n})}|^{\epsilon_n}\Biggr\rvert \nonumber\\
&\leq \frac{1}{2^{k+1}} + \frac{1}{2^{k+1}} + \frac{1}{2^{k+2}} = \frac{5}{4} \frac{1}{2^k} \leq \frac{1}{2^{k-1}}.
\end{align}

Définissons alors $k_{i}(n)$ pour tout $n \in \N$, si $n\leq N^{1}_{i}, k_{i}(n) = 1$ et si $ N_{i}^l \leq n < N_{i}^{l+1},~ k_{i}(n) = l$.

On peut alors définir 

$$y_n = (
\frac{A^{k_{1}(n)}_{1}(x_{1,n}, \cdots, x_{d,n})}{B^{k_{1}(n)}_{1}(x_{1,n}, \cdots, x_{d,n})}, \cdots,\frac{A^{k_{m}(n)}_{m}(x_{1,n}, \cdots, x_{d,n})}{B^{k_{m}(n)}_{m}(x_{1,n}, \cdots, x_{d,n})} ) \\
$$

On doit maintenant montrer la convergence de $(y_n x_n)$ vers $(y,x)$.

\begin{lem}
Soit $1\leq i \leq m$, alors $|\frac{A^{k_{i}(n)}_{i}(x_{1,n}, \cdots, x_{d,n})}{B^{k_{i}(n)}_{i}(x_{1,n}, \cdots, x_{d,n})}|^{\epsilon_n}$ converge vers $\eta_x(y_{i})$. 

En particulier, toutes les valeurs d'adhérence de $(y_n, x_n)$ sont de la forme $(a,x)$ 

où $a \in \ana^{m,an}_{\mathcal{H}(x)}$.
\end{lem}
\begin{proof}[Démonstration du lemme]
Pour $n \geq N^1_{i}$, 
\begin{align*}
\Biggr\lvert \eta_x(m_{i}) - |\frac{A^{k_{i}(n)}_{i}(x_{1,n}, \cdots, x_{d,n})}{B^{k_{i}(n)}_{i}(x_{1,n}, \cdots, x_{d,n})}|^{\epsilon_n} \Biggr\rvert &\leq \Biggr\lvert \eta_x(y_{i}) - \eta_x(y_{i}^{k_{i}(n)})\Biggr\rvert +  \Biggr\lvert \eta_x(y^{k_{i}(n)}_{i}) - |\frac{A^{k_{i}(n)}_{i}(x_{1,n}, \cdots, x_{d,n})}{B^{k_{i}(n)}_{i}(x_{1,n}, \cdots, x_{d,n})}|^{\epsilon_n}\Biggr\rvert\\
&\leq \frac{1}{2^{k_{i}(n)+1}} + \frac{1}{2^{k_{i}(n)+1}} = \frac{1}{2^{k_{i}(n)}} 
\end{align*} 
Ce qui montre la convergence voulue. 

Pour la deuxième partie du lemme, supposons par l'absurde que $(\tilde{y}, \tilde{x})$ est une valeur d'adhérence de $(y_n, x_n)$ où $\tilde{x} \in X^{\beth}$. Comme les coefficients de $y_n$ pris à la puissance $\epsilon_n$ sont bornés, par le résultat précédent, alors comme $\tilde{x} \in X^\beth$ nécessairement $\tilde{y}$ sera aussi dans $\ana^{m,\beth}_k$ ce qui est impossible. 
\end{proof}

On peut maintenant montrer que la suite $(y_n, x_n)$ converge vers $(y,x)$, on montre pour cela que la suite n'a qu'une valeur d'adhérence. 

On utilise le fait que $(y,x)$ est le seul point de $\pr_2^{-1}(x) = \lbrace (a,x) \in \delta (\ana^m \times X)\rbrace$ tel que 
$$\lim\limits_{k \to \infty} \eta_{(a,x)} (T_{i} - \frac{A^k_{i}}{B_{i}^k}) = \lim\limits_{k \to \infty} \frac{\eta_{(a,x)} (T_{i}B_{i}^k - A^k_{i})}{\eta_{(a,x)}(B_{i}^k)} = 0,$$
pour tout $1 \leq i \leq m$. 

On va donc montrer que $\frac{\eta_{(y_n,x_n)} (T_{i}B_{i}^k - A^k_{i})}{\eta_{(y_n,x_n)}(B_{i}^k)} \leq v_k$ pour $n$ assez grand et tel que $v_k \rightarrow 0$. 

Soient $k \in \N, 1 \leq i \leq m, n \geq N^k_{i}$,

\begin{align}\label{Controle A^k(n)/B^k(n) - A^k/B^k}
\frac{\eta_{(y_n,x_n)} (T_{i}B_{i}^k - A^k_{i})}{\eta_{(y_n,x_n)}(B_{i}^k)} &= \Biggr\vert \frac{A_{i}^{k_{i}(n)}(x_{1,n}, \cdots, x_{d,n})}{B_{i}^{k_{i}(n)}(x_{1,n}, \cdots, x_{d,n})} - \frac{ A_{i}^k(x_{1,n}, \cdots, x_{d,n})}{B_{i}^k(x_{1,n}, \cdots, x_{d,n})}\Biggr\rvert^{\epsilon_n}\nonumber\\
&= \Biggr\lvert \sum\limits_{l=k}^{k_{i}(n)-1}\frac{A_{i}^{l+1}(x_{1,n}, \cdots, x_{d,n})}{B_{i}^{l+1}(x_{1,n}, \cdots, x_{d,n})} - \frac{A_{i}^l(x_{1,n}, \cdots, x_{d,n})}{B_{i}^l(x_{1,n}, \cdots, x_{d,n})}\Biggr\rvert^{\epsilon_n}\nonumber\\
&\leq \sum\limits_{l=k}^{k_{i}(n)-1} \frac{1}{2^{l-1}} \nonumber\\
&\leq \sum\limits_{l=k-1}^{\infty} \frac{1}{2^{l}} \nonumber\\ 
&= \frac{1}{2^{k-2}}
\end{align}
Pour la troisième ligne, on utilise l'inégalité \ref{controle diff entre A^k/B^k et A^k+1/B^k+1}.

Donc on en déduit que 
\begin{equation*}
\lim\limits_{k\to \infty} (\limsup\limits_{n\to \infty} \eta_{(y_n, x_n)}(T_{i} -\frac{A^k_{i}}{B_{i}^k} )) = 0.
\end{equation*}
Donc, nécessairement $(y,x)$ est la seule valeur d'adhérence, donc la suite converge vers celle-ci.
\end{proof}
\begin{prop}\label{existence de y_n qui tends vers y pour tout point de type 1}
Soit $x_n \in X(k) \rightarrow x' \in \delta X$ et soit $y \in \ana^{m,an}_{\mathcal{H}(x)}(\overline{\mathcal{H}(x)})$. Alors il existe $y_n \in \ana^{m,an}_k(\overline{k})$ tel que $(y_n, x_n) \rightarrow (y,x')$.
\end{prop}

\begin{rem}
Dans le cas où $x_n \in \delta X$ ou $x_n \in X^{an}\backslash X(k)$, alors les $y_n$ sont des éléments de $\ana^{m,an}_{\mathcal{H}(x_n)}(\overline{\mathcal{H}(x_n)})$.
\end{rem}
\begin{proof}
Notons $x= (x_1, \cdots, x_d), x_n = (x_{1,n}, \cdots, x_{d,n}), y = (y_1, \cdots, y_m)\in K^m$ où $K$ est la clôture algébrique de $\mathcal{H}(x)$. 
\begin{lem}
Il existe $\beta$ algébrique sur $\mathcal{H}(x)$ dont le polynôme minimal est à coefficients dans $Frac(R/ker~\eta_x)$ avec $y \in (\mathcal{H}(x)(\beta))^m$.
\end{lem}
\begin{proof}[Démonstration du lemme]
On sait que $y \in (\mathcal{H}(x)(y_{1}, \cdots, y_{m}))^m$, donc par le théorème de l'élément primitif, il existe $\alpha$ algébrique sur $\mathcal{H}(x)$ tel que $\mathcal{H}(x)(y_{1}, \cdots, y_{m}) = \mathcal{H}(x)(\alpha)$.

Soit $\epsilon > 0$, notons $\mu_\alpha = \sum_{k=0}^l \mu_{\alpha,k}T^k$ le polynôme minimal de $\alpha$ et $\alpha_i$ ses racines. Alors il existe $P = \sum_{k=0}^l P_k T^K \in Frac(R/ker~\eta_x)[T]$ tel que $max(\tilde{\eta_x}(P_k - \mu_{\alpha,k})) \leq \epsilon$ où $\tilde{\eta_x}$ est la norme sur $K$ qui prolonge $\eta_x$. Si l'on prends $\mu_\alpha$ unitaire, on peut aussi prendre $P$ unitaire et ainsi si $\beta$ est une racine de $P$, on a :

\begin{equation*}
\tilde{\eta_x}(\beta) \leq max(\eta_x(P_0), \cdots, \eta_x(P_{l-1}), 1) \leq  max(\tilde{\eta_x}(\mu_{\alpha,0}), \cdots, \tilde{\eta_x}(\mu_{\alpha,l-1}), 1) +\epsilon :=C.
\end{equation*}

D'où
\begin{equation*}
C^l\epsilon \geq \tilde{\eta_x}((P-\mu_{\alpha})(\beta)) = \tilde{\eta_x}(\mu_{\alpha}(\beta)) = \prod_{\alpha_i~racines~de~\mu_\alpha} \tilde{\eta_x}(\alpha_i-\beta). 
\end{equation*}
Donc, l'un des $\alpha_i$ vérifie que $\tilde{\eta_x}(\alpha_i-\beta) \leq C \epsilon^{\frac{1}{l}}$. En prenant $\epsilon$ tel que $\epsilon < (\frac{\max (\tilde{\eta_x}(\alpha_i -\alpha_j))}{C})^l$, on obtient un unique $\alpha_i$ vérifiant cette condition et telle que $\tilde{\eta_x}(\alpha_i-\beta) < \min_{i\neq j}(\tilde{\eta_x}(\alpha_j-\beta))$.

Donc par le lemme de Krasner, on a $\mathcal{H}(x)(\alpha) \subset \mathcal{H}(x)(\beta)$. Comme le polynôme $P$ est annulateur de $\beta$, on a $\deg(\mu_\beta) \leq \deg(P_\beta) = \deg(\mu_\alpha)$, où $\mu_\beta$ désigne un polynôme minimal de $\beta$. Par l'inclusion ci-dessus, on a  $\deg(\mu_\beta) \geq \deg(\mu_\alpha)$ et donc $P$ est en fait un polynôme minimal de $\beta$. 
\end{proof}
Donc, $y \in (\mathcal{H}(f)(\beta))^m$ où $\beta$ est un élément algébrique sur $\mathcal{H}(x)$ ayant un polynôme minimal à coefficients dans $\Frac(R/\ker~\eta_x)$.

Comme $y \in (\mathcal{H}(x)(\beta))^m$ il existe $P_{1}, \cdots, P_{m} \in \mathcal{H}(x)[T]$ tel que $y_{i} = P_{i}(\beta)$. 

Notons $P_{i} (\beta) = \sum\limits_{l=0}^{D_{i}} a_{i,l} \beta^l$ avec $a_{i,l} \in \mathcal{H}(x)$. Comme à la proposition \ref{existence convergence vers y dans H(x)^m} ,on construit $a_{i,l}^k := \frac{A_{i,l}^k}{B_{i,l}^k}$ où $A_{i,l}^k, B_{i,l}^k \in R$ tel que $\eta_x(a_{i,l}^k - a_{i,l}) \leq \frac{1}{2^{k+1}}$. De plus, par la proposition \ref{existence de beta_n qui tends vers beta} on sait qu'il existe $\beta_n \in \overline{k}$ tel que $(\beta_n, x_n) \rightarrow (\beta, x) \in (\ana^1\times X)^\urcorner$. 

Définissons alors 
\begin{equation*}
y^{(k)} = (\sum\limits_{l=0}^{D_{1}} a_{1,l}^k \beta^l, \cdots, \sum\limits_{l=0}^{D_{m}} a_{m,l}^k \beta^l  , y_n^{(k)} = (\sum\limits_{l=0}^{D_{1}} \frac{A_{1,l}^k(x_{1,n}, \cdots, x_{d,n})}{B_{1,l}^k(x_{1,n}, \cdots, x_{d,n})}\beta_n^l, \cdots, \sum\limits_{l=0}^{D_{m}} \frac{A_{m,l}^k(x_{1,n}, \cdots, x_{d,n})}{B_{m,l}^k(x_{1,n}, \cdots, x_{d,n})}\beta_n^l).
\end{equation*}
Par construction de $\beta_n$ et comme tout est polynomial on a $(y_n^{(k)}, x_n) \rightarrow (y^{(k)}, x)$, on peut le faire explicitement avec les arguments de la preuve de la proposition \ref{existence suite de points qui converge vers y dans ext alg de Frac}.

Enfin définissons, 
\begin{equation*}
y_n = (\sum\limits_{l=0}^{D_{1}} \frac{A_{1,l}^{k_{1}(n)}(x_{1,n}, \cdots, x_{d,n})}{B_{1,l}^{k_{1}(n)}(x_{1,n}, \cdots, x_{d,n})}\beta_n^l, \cdots, \sum\limits_{l=0}^{D_{m}} \frac{A_{m,l}^{k_{m}(n)}(x_{1,n}, \cdots, x_{d,n})}{B_{m,l}^{k_{m}(n)}(x_{1,n}, \cdots, x_{d,n})}\beta_n^l)
\end{equation*}
où les $k_{i}(n)$ sont définis comme précédemment dans la preuve de la proposition \ref{existence convergence vers y dans H(x)^m}.

Montrons maintenant $(y_n, x_n) \rightarrow (y,x)$. 

Soit $P \in A =k[T'_1, \cdots, T'_m, T_1, \cdots, T_d]$ et soit $\epsilon >0$. 

Comme $P(y_1, \cdots, y_m, x_1, \cdots, x_d)$ est un élément de $\mathcal{H}(x)(\beta)$ et $y^{(k)} \rightarrow y$ coefficient par coefficient dans $(\mathcal{H}(f)(\beta))^m$, donc il existe $K \in \N$ tel que $\forall k \geq K,$ 
\begin{equation*}
\tilde{\eta_x}(P(y_1, \cdots, y_m, x_1, \cdots, x_d) - P(y^{(k)}_1, \cdots, y^{(k)}_m, x_1, \cdots, x_d))  \leq \epsilon.
\end{equation*}
Donc, 
\begin{equation*}
|\eta_{(y,x)}(P) - \eta_{(y^{(k)}, x)}(P)| \leq \epsilon.
\end{equation*}

\begin{lem}
Soient $1\leq i \leq m, k \geq 2, n \geq N^k_{i}$ alors il existe une constante $C_{i}$ ne dépendant ni de $k$, ni de $n$ tel que 
\begin{equation*}
|y_{i,n} - y^{(k)}_{i,n}|^{\epsilon_n} \leq \frac{C_{i}}{2^{k-2}}.
\end{equation*}
\end{lem}
\begin{proof}Démonstration du lemme

On a :
\begin{align*}
|y_{i,n} - y^{(k)}_{i,n}|^{\epsilon_n} &= |\sum\limits_{l=0}^{D_{i}} (\frac{A_{i,l}^{k_{i}(n)}}{B_{i,l}^{k_{i}(n)}} - \frac{A_{i,l}^{k}}{B_{i,l}^{k}}) (x_{1,n}, \cdots, x_{d,n})\beta_n^l|^{\epsilon_n}\\
&\leq \frac{1}{2^{k-2}} \sum\limits_{l=0}^{D_{i}} |\beta_{n}^l|^{\epsilon_n}\\
&\leq \frac{C_{i}}{2^{k-2}},
\end{align*}
où la deuxième égalité vient de la démonstration précédente à l'inégalité \ref{Controle A^k(n)/B^k(n) - A^k/B^k} et la dernière vient du fait que $|\beta_{n}^l|^{\epsilon_n} \rightarrow \tilde{\eta_x}(\beta^l)$ et la suite est donc bornée.
\end{proof}
Ainsi,
\begin{equation*}
|P(y_{1,n}, \cdots, y_{m,n}, x_{1,n}, \cdots, x_{d,n})|^{\epsilon_n} = |P(y^{(k)}_{1,n}, \cdots, y^{(k)}_{m,n}, x_{1,n}, \cdots, x_{d,n})|^{\epsilon_n} + \Omega
\end{equation*}
où $\Omega \leq cste \frac{1}{2^{k-2}}$. En prenant $k$ assez grand, $\frac{1}{2^{k-2}} \leq \epsilon$, d'où :
\begin{equation*}
||P(y_{1,n}, \cdots, y_{m,n}, x_{1,n}, \cdots, x_{d,n})|^{\epsilon_n} - \eta_{(y,x)}(P)| \leq (2 + cste) \epsilon.
\end{equation*}
Ce qui montre la convergence de $(y_n, x_n)$ vers $(y,x)$.
\end{proof}

On va maintenant utiliser cela pour montrer que pour tout $k$-schéma $Y$ affine de type fini, si $x_n \in X(k) \rightarrow x \in \delta X$ et $y \in Y^{an}(\overline{\mathcal{H}(x)})$, alors il existe $y_n \in Y(\overline{k})$, tel que $(y_n, x_n) \rightarrow (y,x)$.

\begin{defi}\label{defi normalisation noether}
Soit $Y$ un $k$-schéma affine de type fini. Alors notons $\pi^{noether}$ le morphisme surjectif fini $Y \rightarrow \ana^m_k$ provenant du lemme de normalisation de Noether. Et on notera $\pi^{noether, hyb}$ le morphisme entre $Y^{hyb}$ et $\ana_k^{m,hyb}$.
\end{defi}

\begin{prop}\label{Existence suite convergente pour tout schéma}
Soit $x_n \in X(k) \rightarrow x \in \delta X$ et $Y$ un $k$-schéma de type fini. Soit $y \in Y_{\mathcal{H}(x)}(\overline{\mathcal{H}(x)})$, alors il existe $y_n \in Y^{an}$ tel que $(y_n, x_n) \in (Y\times X)^{an} \rightarrow (y,x) \in \delta (Y\times X)$.
\end{prop}

\begin{rem}
Dans le cas où $x_n \in \delta X$ ou $x_n \in X^{an}\backslash X(k)$, on a $y_n \in Y^{an}_{\h(x_n)}$.
\end{rem}

\begin{proof}




On peut supposer $Y$ affine comme l'on regarde des propriétés locales.

Soit $Y$ un $k$-schéma affine, de type fini $y \in Y_{\mathcal{H}(x)}(\overline{\mathcal{H}(x)})$. Le morphisme de normalisation de Noether est : $\pi^{noether} : Y \rightarrow  \ana^{m}_k$.

Soit $(a,x) = \pi^{noether,hyb}\times id(y,x) \in \ana^{m,hyb}_k \times X^{hyb}$, alors par la proposition \ref{existence convergence vers y dans H(x)^m}, il existe $(a_n, x_n)\in \ana^{m,hyb}_k \times X^{hyb} \rightarrow (a,x)$. Comme le morphisme $\pi^{noether, hyb}$ est surjectif, $(\pi^{noether,hyb}\times id)^{-1}( a_n, x_n) \neq \emptyset$, le but est donc de trouver des $(y_n, x_n) \in (\pi^{noether,hyb}\times id)^{-1}( a_n, x_n)$ tel que $(y_n, x_n) \rightarrow (y,x)$. 

Comme le morphisme $\pi^{noether,hyb}\times id$ est fini, il existe $y_1, \cdots, y_l$ tous distincts et différent de $y$ tel que $(\pi^{noether,hyb}\times id)^{-1}(\pi^{noether,hyb}\times id)(y,x) = \lbrace (y,x), (y_1,x), \cdots, (y_l,x)\rbrace$. De plus, pour tout $i$, il existe $y \in U_i \subset Y^{hyb}\times X^{hyb}, y_i \in V_i \subset Y^{hyb}\times X^{hyb}$ deux ouverts d'intersection vide. Prenons alors $U =\bigcap U_i$, on a donc $y \in U$ et pour tout $i, V_i \cap U = \emptyset$.

Comme le morphisme $\pi^{noether,hyb}\times id$ est fini et $Y$ et $\ana^m_k$ sont des schémas de même dimension, il est ouvert par la proposition \ref{morphisme fini sur espace hybride est ouvert}. Donc $(a,x) \in \pi^{noether,hyb}\times id(U)$ qui est ouvert donc il existe $N \in \N, \forall n \geq N, (a_n, x_n)\in \pi^{noether,hyb}\times id(U)$. Soient donc $(y_n, x_n) \in U \cap (\pi^{noether,hyb}\times id)^{-1}(a_n, x_n)$ pour tout $n \geq N$ et pour $n<N$, on prends $(y_n, x_n) \in (\pi^{noether,hyb}\times id)^{-1}(a_n, x_n)$ quelconque. 

Alors, comme $\pi^{noether,hyb}\times id (y_n,x_n) \rightarrow (a,x)$, nécessairement les valeurs d'adhérence de cette suite sont dans $ \lbrace (y,x), (y_1,x), \cdots, (y_l,x)\rbrace$. Or pour $n \geq N, (y_n, x_n) \notin V_i$, donc $y_i$ ne peut pas être une valeur d'adhérence. Donc $(y_n, x_n) \rightarrow (y,x)$.
\end{proof}

\begin{rem}
On peut remplacer des suites par des filets i.e. des suites de Moore-Smith et tous les raisonnements fonctionneraient exactement de la même façon, on travaillera donc avec des suites ou des filets dans la suite. 
\end{rem}

\section{L'action sur le bord}\label{SectionResultatActionBord}

Le but de cette partie est de trouver un lieu du bord où l'action de $G$ est bien définie et un lieu où elle n'est pas bien définie. Ensuite, le but est d'étudier les liens entre la compactification de $X/G$ et un quotient de la compactification de $X$ par l'action de $G$, en trouvant certains cas où l'on dispose d'une bijection continue entre les 2.

\subsection{Lieu de bonne définition de l'action sur le bord}

Soit $X$ une $k$-variété et $G$ un $k$-groupe algébrique agissant sur $X$. Supposons que $X$ peut être recouvert par des schémas affines $G$-invariants.
 
\begin{defi} 
L'action de $G$ sur $X$ est défini par un morphisme 
\begin{center}
$ \Phi : X \times G \rightarrow X$ 
\end{center} 
qui s'étend à l'analytification hybride : $\Phi^{hyb} : (X \times G)^{hyb} \rightarrow X^{hyb}$. Ainsi sur chaque point de $x \in X_\infty$, l'on dispose d'une action de $G^{an}_{\mathcal{H}(x)}:= \pr_2((\Phi^{hyb})^{-1}(x))$ sur $x$. On dira que l'action d'un sous-ensemble $H \subset G^{an}_{\mathcal{H}(x)} $ est bien définie en le point $x$ vu dans $X^\urcorner$ si et seulement si $\forall g\in H, g\cdot x \in X_\infty$. 
\end{defi}
 
\begin{prop}\label{si pas degenerence dans X/G alors action pas bien def}
Soit $x_n \in X^{an} \rightarrow x \in \delta X$. Supposons qu'il existe $g_n \in G^{an}$ tel que $(g_n\cdot x_n) \in X^{an}$ ait une valeur d'adhérence dans $X^{an}$. Alors il existe $g \in \delta G$ tel que $g\cdot x \in X^\beth$ i.e. l'action de $G^{an}_{\mathcal{H}(x)}$ sur $X^\urcorner$ n'est pas bien définie en le point $x$. 
\end{prop}

\begin{proof}
Posons $y_n := g_n \cdot x_n$ alors quitte à extraire on peut supposer que $y_n \rightarrow y \in X^{an}$. 

Soit $z \in (G \times X)\urcorner$ une valeur d'adhérence de $(g_n^{-1}, y_n)$. On note $a = \pr_2(z) \in X^{hyb}$ et $g^{-1} \in G^{an}_{\mathcal{H}(y)}$ tel que $(g^{-1}, a) = z$. Comme $y_n \rightarrow y \in X^{an}$, nécessairement $a = y$ soit vu comme un élément de $X^{an}$ si $(g^{-1}, a) \in (G\times X)^{an}$ soit vu comme un élément de $X^\beth$ si $(g^{-1}, a) \in \delta (G\times X)$. 

Puisque $(g^{-1}, y)$ est une valeur d'adhérence de $(g_n^{-1}, y_n)$, par la proposition \ref{va de fg_n x_n est va de g x}, $g^{-1} \cdot y$ est une valeur d'adhérence de $g_n^{-1}\cdot y_n= x_n \rightarrow x$. Donc $x =g^{-1} \cdot y $ et donc $ g\cdot x =y$. Comme $x\in \delta X, (g^{-1},y) \in \delta (G\times X)$ donc $y \in X^\beth$ et donc $g\cdot x$ aussi. 
\end{proof}

\begin{prop}\label{bonne déf action}
Soit $x_n \in X^{an} \rightarrow x\in \delta X$. Supposons que $\forall g_n \in G^{an}, g_n\cdot x_n$ n'a pas de valeur d'adhérence dans $X^{an}$. Alors pour tout $g \in (G\times \mathcal{H}(x))^{an}(\overline{\mathcal{H}(x)})$, $g\cdot x$ définit un point dans le bord de $X$, donc l'action de $G^{an}_{\mathcal{H}(x)}(\overline{\mathcal{H}(x)})$ est bien définie en ce point $x$. 
\end{prop}

\begin{proof}
Soit $g \in G^{an}_{\mathcal{H}(x)}(\overline{\mathcal{H}(x)})$.

On sait par la proposition \ref{existence de y_n qui tends vers y pour tout point de type 1} qu'il existe $g_n \in G^{an}$ tel que $(g_n, x_n) \rightarrow (g,x)$ et donc $(g_n\cdot x_n)_{\epsilon_n} \rightarrow g\cdot x \in X^{hyb}$, par la proposition \ref{g_n x_n tend vers g x} où $(g_n\cdot x_n)_{\epsilon_n}$ signifie que l'on regarde le point correspondant à $(g_n\cdot x_n)$ dans la fibre $pr^{-1}(\epsilon_n) \in X^{hyb}$. 

Or, par hypothèse $g_n\cdot x_n$ n'a pas de valeur d'adhérence dans $X^{an}$ donc nécessairement, en voyant $g_n\cdot x_n \in X^\urcorner$, toutes ses valeurs d'adhérence sont dans $\delta X$ et donc $g\cdot x \in \delta X$.
\end{proof}

Supposons maintenant de plus que le groupe $G$ est réductif. Dans \cite{GIT}, Mumford définit plusieurs deux sous-schémas $X^s(\Pre)$ et $X^s(L)$ de $X$ sur lesquels on peut définir le quotient géométrique de ces sous-schémas par l'action de $G$ que l'on note $X^s(\Pre)/G, X^s(L)/G$. 

\begin{rem}
On peut aussi définir $X^{ss}(L)$ le lieu semi-stable et définir le quotient catégorique $X^{ss}(L)/\! /G$.
\end{rem}

\begin{defi}\label{Définition lieu stable}
Soit $X$ une $k$-variété et $G$ un groupe algébrique réductif agissant sur $X$. Soit $x$ un point géométrique de $X$. On dit que 
\begin{itemize}
\item $x$ est pré-stable s'il existe un ouvert affine $U$ invariant par $G$ tel que $x\in U$ et l'action de $G$ sur $U$ est fermée.

Soit $L$ un faisceau inversible sur $X$ et $\phi$ une $G$-linéarisation de $L$. On pourra se référer au §3 du chapitre 1 de \cite{GIT} pour une définition de $G$-linéarisation. Alors,
\item $x$ est stable (vis à vis de $L, \phi$) s'il existe une section $s \in H^0(X,L^n)$ pour un certain $n$ tel que $s(x) \neq 0$, $X_s$ est affine, $s$ est invariant et l'action de $G$ sur $X_s$ est fermé. 
\end{itemize}

Alors, l'ensemble des points géométriques vérifiant l'une de ces propriétés est l'ensemble des points d'un ouvert de $X$ que l'on notera respectivement : 
\begin{align*}
&X^{s}(\Pre)\\
&X^s(L).
\end{align*}
\end{defi}

\begin{rem}\label{Cas affine, toute la variete est lieu stable}
Si $X$ est une $k$-variété affine tel que l'action est fermée, alors il existe un faisceau inversible $L$ sur $X$ tel que $X = X^s(L)$, on pourra se référer au converse 1.12 du chapitre 1 de \cite{GIT}.
\end{rem}

On peut alors énoncer une version du théorème de Mumford (GIT) dans \cite{GIT}.
\begin{thm}\label{GIT alébrique}
Soit $\mathcal{X}$ une $k$-variété et notons $X = \mathcal{X}^s(\Pre)$ ou $X = \mathcal{X}^s(L)$ pour $L$ un faisceau inversible.

Alors, le quotient géométrique $X/G =:Y$ existe en tant que schéma sur $k$.  

Cela signifie que l'on dispose d'un morphisme $\pi : X \rightarrow Y$ tel que si l'on note $\sigma : G\times_k X \rightarrow X$ l'action de $G$ alors :
\begin{itemize}
\item On a : $\pi \circ \sigma : G\times X \rightarrow Y = \pi \circ pr_2$,
\item $\pi$ est surjective et l'image de $\Phi = (\sigma, pr_2): G\times_Y X \rightarrow X\times_Y X$ est $X\times_Y X$, ce qui est équivalent au fait que les fibres géométriques de $\pi$ sont les orbites des points géométriques de $x$,
\item $\pi$ est une submersion i.e. $U \subset Y$ est ouvert ssi $\pi^{-1}(U) \subset X$ l'est.
\end{itemize}

De plus, si $\mathcal{X}$ est affine et l'action est fermée, alors $Y$ est un schéma affine de type fini sur $k$. De plus, en notant $R = \Gamma(X, \mathcal{O}(X))$, alors $Y = \Spec~R^G$ où $R^G$ désigne les éléments invariants par $G$.

Si $X = \mathcal{X}^s(L)$ pour $L$ un faisceau inversible, alors $Y$ est quasi-projectif sur $k$ et donc en particulier une $k$-variété. De plus, $\pi$ est affine.
\end{thm}

Dans le cas affine, on a un résultat dû à M. Maculan, \cite{MaculanGIT} qui montre ce théorème dans le cas des espaces analytiques. 
\begin{thm}\label{GIT analytique} Proposition 3.1 et 3.8 de \cite{MaculanGIT}

Soit $X$ un $k$-schéma affine de type fini. Si l'on analytifie $X$ et $X/G$ selon la valeur absolue de $k$ alors le morphisme analytifié $\pi^{an}$ vérifie :
\begin{itemize}
\item $\pi^{an} : X^{an} \rightarrow (X/G)^{an}$ est surjectif et $G$-invariant.
\item Pour tout $x, x' \in X^{an},$
\begin{center}
$\pi^{an}(x) = \pi^{an}(x') \iff \overline{G^{an}_{\mathcal{H}(x)} \cdot x} \cap \overline{G^{an}_{\mathcal{H}(x')}\cdot x'} \neq 0.$
\end{center}
\item Pour tout $x \in X^{an}$, il existe une unique orbite fermée contenue dans $\overline{G^{an}_{\mathcal{H}(x)} \cdot x}$.
\item $\pi^{an}$ est une submersion.
\end{itemize}

En particulier, si l'action est fermée, alors pour tout $x, x' \in X^{an}$,
\begin{center}
$\pi^{an}(x) = \pi^{an}(x') \iff G^{an}_{\mathcal{H}(x)}\cdot x = G^{an}_{\mathcal{H}(x')}\cdot x'.$
\end{center} 
\end{thm}

\begin{nota}
On notera souvent $G^{an}$ pour parler de $G^{an}_{\mathcal{H}(x)}$. 
\end{nota}

\begin{nota}\label{Notation du lieu stable}
Dans la suite, on écrira $\mathcal{X}$ pour désigner une $k$-variété. On prend un couple $(\mathcal{X}, X)$ pour désigner l'un des 3 cas suivants :
\begin{itemize}
\item $X$ est le lieu stable $\mathcal{X}^s(L)$ où $L$ est un faisceau inversible sur $\mathcal{X}$,
\item $\mathcal{X}$ est affine et l'action de $G$ est fermée. Alors, dans ce cas on prend $X = \mathcal{X}$. Par la remarque \ref{Cas affine, toute la variete est lieu stable}, c'est un cas particulier du premier cas,
\item $X$ est le lieu pré-stable $X =\mathcal{X}^s(\Pre)$ et $\mathcal{X}^s(\Pre)/G$ est une $k$-variété.  
\end{itemize}
\end{nota}
\begin{rem}
En général, $\mathcal{X}^s(\Pre)/G$ n'est pas forcément une $k$-variété et c'est donc une hypothèse du 3ème cas.
\end{rem}

\begin{prop}\label{existence convergence fibre si convergence dans le quotient}
Soit $\mathcal{X}$ une $k$-variété, $G$ un groupe algébrique réductif agissant sur $\mathcal{X}$ et notons $X$ comme dans la notation \ref{Notation du lieu stable}.

Notons $\pi^{an} : X^{an} \rightarrow (X/G)^{an}$ la projection. Soient $x_n \in (X^{an})^\N$ et supposons que $\pi^{an}(x_n) \rightarrow y \in (X/G)^{an}$. Soit $x \in X^{an}$ tel que $\pi^{an}(x) = y$. 

Alors quitte à extraire il existe $g_n \in G^{an}$, tel que $g_n\cdot x_n \rightarrow x.$ 
\end{prop}
\begin{proof}
La preuve s'appuie sur les idées de la démonstration de la proposition 2.2 de Favre-Gong \cite{FavreGong}.

Par la définition de $\mathcal{X}^s(L), \mathcal{X}^s(\Pre)$, on peut se ramener au cas où $X$ est affine et l'action de $G$ sur $X$ est fermée.

Soit $x \in X^{an}$ tel que $\pi^{an}(x) = y$. Le but est de montrer qu'il existe $g_n \in G^{an}$ tel que, quitte à extraire, $g_n \cdot x_n \rightarrow x$. L'existence des $g_n$ est immédiate si $\pi^{an}(x_n) = y$ une infinité de fois, on peut donc supposer que $\forall n, \pi^{an}(x_n) \neq y$. Posons alors $\mathcal{A} := \bigcup_{n\in \N} (\pi^{an})^{-1}(\pi^{an}(x_n))$.

Alors le but est de montrer que $\mathcal{A}$ n'est pas fermé. En effet, si $\mathcal{A}$ n'est pas fermé, il existe $\alpha \notin \mathcal{A}$ adhérent à $\mathcal{A}$. Comme $\alpha\notin \mathcal{A}$ et que les fibres sont fermées, si on prend une suite $\alpha_k \in \mathcal{A}$ qui tend vers $\alpha$, alors il n'y a qu'un nombre fini de $x_k$ dans chaque $\pi^{-1,an}(\pi^{an}(x_n))$. Ainsi $\forall k, \exists n(k), \exists g_{n(k)} \in G^{an}, \alpha_k = g_{n(k)}\cdot x_{n(k)}$ et quitte à extraire les $\alpha_k$, on peut supposer que $n(k)$ est strictement croissant. De plus, comme $\pi(\alpha_k) = \pi(x_{n(k)}) \rightarrow y$, on en déduit que $\pi(\alpha) = y$. Donc toutes les valeurs d'adhérence de $\alpha_k$ sont dans $\pi^{-1,an}(y)$. 

Donc, il existe $g\in G^{an}$ tel que $\alpha = g\cdot x$. Comme $\mathcal{H}(x)$ n'est pas trivialement valué, $G^{an}_{\mathcal{H}(x)}(\overline{\mathcal{H}(x)})$ est dense dans $G^{an}_{\mathcal{H}(x)}$ et comme les espaces de Berkovich sur un corps sont angéliques, \cite{PoineauAngelique}, il existe $h_n \in G^{an}_{\mathcal{H}(x)}(\overline{\mathcal{H}(x)})$ tel que $h_n \rightarrow g^{-1}$. Par la proposition \ref{Existence suite convergente pour tout schéma}, il existe $h_{n,k} \in G^{an}_{\mathcal{H}(\alpha_k)}$ tel que $(\alpha_k, h_{n,k}) \rightarrow (g \cdot x, h_n) \in (X\times G)^{an}$. 

Alors, $(g\cdot x, g^{-1})$ est adhérent à $\lbrace (\alpha_k, h_{n,k}), (n,k) \in \N^2\rbrace$. En réutilisant l'angélicité des espaces de Berkovich, quitte à extraire les $\alpha_k$, il existe $h_k \in G^{an}$ tel que $(\alpha_k, h_k) \rightarrow (g\cdot x, g^{-1})$ et donc $h_k\cdot \alpha_k \rightarrow x$.

Il reste donc à montrer que $\mathcal{A}$ n'est pas fermé. Cela vient du fait que par le théorème \ref{GIT analytique}, si $\mathcal{A}$ est fermé, alors $\pi^{an}(\mathcal{A})$ est aussi fermé et donc $y \in \pi^{an}(\mathcal{A})$. Or par définition de $\mathcal{A}$, on a $\mathcal{A} = \pi^{an}((\pi^{an})^{-1}(\mathcal{A}))$ et donc $\pi^{-1,an}(y) \subset \mathcal{A}$ ce qui contredit le fait que $\forall n, \pi^{an}(x_n) \neq y$.
\end{proof}

\begin{prop}\label{lien dégénéresnce quotient et valeur d adherence des fibres}
Soit $\mathcal{X}$ une $k$-variété, $G$ un groupe algébrique réductif agissant sur $\mathcal{X}$ et notons $X$ comme dans la notation \ref{Notation du lieu stable}.

Soit $x_n \in (X^{an})^\N$ et notons $\pi^{an} : X^{an} \rightarrow (X/G)^{an}$ la projection, alors
\begin{center}
$\forall g_n \in G^{an}, g_n\cdot x_n$ n'a pas de valeurs d'adhérence dans $X^{an} \iff \pi^{an}(x_n) \rightarrow \infty$
\end{center} 
où $\pi^{an}(x_n) \rightarrow \infty$ signifie que cette suite n'a pas de valeur d'adhérence dans $(X/G)^{an}$.
\end{prop}

\begin{proof}
$\Leftarrow$ Si l'on suppose que $g_n \cdot x_n$ a une valeur d'adhérence $x$ dans $X^{an}$, alors $\pi(x)$ est une valeur d'adhérence de $\pi(x_n)$.

$\Rightarrow$ Si l'on suppose que $\pi(x_n)$ a une valeur d'adhérence $y$, alors on peut supposer que $\pi(x_n) \rightarrow y$, car les espaces de Berkovich sur un corps sont angéliques, voir \cite{PoineauAngelique}. Puis par la proposition \ref{existence convergence fibre si convergence dans le quotient}, il existe $g_n \in G^{an}$ tel que $g_n\cdot x_n$ ait une valeur d'adhérence dans $X^{an}$.  
\end{proof}

\begin{prop}\label{Action bien definie suffit sur les points rigides}
Soit $x\in \delta X$. Alors, on a l'équivalence suivante :
\begin{center}
L'action de $G_{\mathcal{H}(x)}^{an}$ est bien définie en $x\iff$ l'action de $G^{an}_{\mathcal{H}(x)}(\overline{\mathcal{H}(x)})$ est bien définie en $x$.
\end{center}
\end{prop}

\begin{proof}
$\Rightarrow$ C'est clair;

$\Leftarrow$ Par définition de $X \subset \mathcal{X}$, $x\in \delta U$ où $U$ est un ouvert affine $G$-invariant et on peut donc se ramener au cas où $X$ est affine.

Comme $X$ est un $k$-schéma affine de type fini, il existe $f_1, \cdots, f_d$ tel que les $f_i$ engendrent $\mathcal{O}(X)$, alors un point $y\in X_0^{an}$ est un point de $X^\beth$ si et seulement si $\forall i, |f_i(y)| \leq 1$. 

On peut se ramener au cas où $G$ est affine, comme les propriétés sont locales. Alors l'action $\Phi :G \times X \rightarrow X$ induit un morphisme $\phi : \mathcal{O}(X) \rightarrow \mathcal{O}(G)\otimes \mathcal{O}(X)$. 

Alors pour tout $i$, pour tout $g \in G_{ \mathcal{H}(x)}^{an}$, on a $|f_i(g\cdot x)| = |\phi(f_i)(g,x)|.$

Soit  $g\in G_{ \mathcal{H}(x)}^{an}$, on a $g\cdot x \in X^\beth$ si et seulement si, $\forall i, |\phi(f_i)(g,x)|\leq 1$.

Supposons par l'absurde qu'il existe un tel $g$. 

Alors comme $G^{an}_{\mathcal{H}(x)}$ est un espace $\mathcal{H}(x)$-analytique, il existe un  voisinage $V$ de $g$ où $V$ est un domaine $\mathcal{H}(x)$-affinoïde, que l'on peut supposer strictement affinoïde comme $\mathcal{H}(x)$ n'est pas trivialement valué. Alors $g \in U$ où $U$ est le domaine strictement affinoïde contenu dans $V$ défini par les équations $|f_i(y)|\leq 1$. Or comme $\mathcal{H}(x)$ n'est pas trivialement valué, par la proposition 2.1.15 de Berkovich (\cite{BerkovichLivre}), on en déduit que $U$ possède un point rigide et donc il existe un point de $G^{an}_{\mathcal{H}(x)}(\overline{\mathcal{H}(x)})$ dans $U$ ce qui est absurde par hypothèse.
\end{proof}

En combinant, les propositions \ref{si pas degenerence dans X/G alors action pas bien def}, \ref{bonne déf action}, \ref{lien dégénéresnce quotient et valeur d adherence des fibres} et \ref{Action bien definie suffit sur les points rigides}, on obtient le théorème suivant.

\begin{thm}\label{equivalence dégénérecence quotient et action bien déf}
Soit $\mathcal{X}$ une $k$-variété, $G$ un groupe algébrique réductif agissant sur $\mathcal{X}$ et notons $X$ le lieu stable vis à vis d'un faisceau inversible sur $\mathcal{X}$ comme dans la notation \ref{Notation du lieu stable}.

Soit $x_n \in (X^{an})^\N$ et notons $\pi^{an} : X^{an} \rightarrow (X/G)^{an}$ la projection. Supposons que $x_n \rightarrow x \in X^\urcorner$ avec $x\in \delta X$, alors
\begin{center}
L'action de $G_{\mathcal{H}(x)}^{an}$ est bien définie en $x \iff \pi^{an}(x_n) \rightarrow \infty$
\end{center} 
$\pi^{an}(x_n) \rightarrow \infty$ signifie que cette suite n'a pas de valeur d'adhérence dans $(X/G)^{an}$.
\end{thm}

\begin{cor}
Soit $X$ un $k$-schéma affine et $G$ un groupe algébrique réductif dont l'action est fermée.

Soit $x_n \in (X^{an})^\N$ et notons $\pi^{an} : X^{an} \rightarrow (X/G)^{an}$ la projection. Supposons que $x_n \rightarrow x \in X^\urcorner$ avec $x\in \delta X$, alors
\begin{center}
L'action de $G_{\mathcal{H}(x)}^{an}$ est bien définie en $x \iff \pi^{an}(x_n) \rightarrow \infty$
\end{center} 
$\pi^{an}(x_n) \rightarrow \infty$ signifie que cette suite n'a pas de valeur d'adhérence dans $(X/G)^{an}$.
\end{cor}

\begin{prop}\label{propreté morphisme sur le quotient}
Soient $X, Y$ deux $k$-schémas affines de type fini, $G$ un groupe réductif, agissant sur $X$ et $Y$ dont l'action est fermée.

Soit $f : X \rightarrow Y$ un morphisme $G$-invariant et supposons que $f^{hyb,-1}(Y^\beth) \subset X^\beth$, alors le morphisme induit $\overline{f} : (X/G)^{an}  \rightarrow (Y/G)^{an}$ est propre.  
\end{prop}
\begin{proof}
Supposons par l'absurde qu'il existe $z_n \in (X/G)^{an}$ tel que $\overline{f}(z_n) \rightarrow z \in (Y/G)^{an}$ et $z_n \rightarrow \infty$.

Prenons $x_n\in X^{an}$ tel que $\pi_X(x_n) = z_n$, posons alors $y_n = f(x_n) \in Y^{an}$, donc $\pi_Y(y_n) = \overline{f}(z_n) \rightarrow z$.

Soit $y\in Y^{an}$ tel que $\pi_Y(y) = z$, alors quitte à extraire il existe $g_n\in G^{an}, g_n\cdot y_n = f(g_n\cdot x_n) \rightarrow y$, par la proposition \ref{existence convergence fibre si convergence dans le quotient}. Or $\pi_X(x_n) \rightarrow \infty$, donc par la proposition \ref{lien dégénéresnce quotient et valeur d adherence des fibres} $g_n\cdot x_n$ n'a pas de valeur d'adhérence dans $X^{an}$. Soit $x$ une de ces valeurs d'adhérence dans $X^\urcorner$. Alors $x\in \delta X$. Il existe donc $\epsilon_n \rightarrow 0$, tel que $\eta_x \in X^{hyb}$ soit une valeur d'adhérence de $\eta_{g_n\cdot x_n}^{\epsilon_n}$, par la proposition \ref{existence de releve pour les suites convergentes}. Alors $f^{hyb}(\eta_x)$ est une valeur d'adhérence de $f^{hyb}(\eta_{g_n \cdot x_n}^{\epsilon_n})$. Or $f(g_n \cdot x_n)\rightarrow y \in Y^{an}$, donc nécessairement il existe $\alpha >0$ avec $f^{hyb}(x) = y^\alpha$ où $y \in Y^\beth$, donc $x\in X^\beth$ ce qui est absurde.
\end{proof}

\subsection{Lien entre le quotient de la compactification et compactification du quotient}

On va maintenant comparer le quotient de la compactification sur le lieu où l'action est bien définie et la compactification du quotient.

\begin{defi}\label{définition lieu ou action est bien def et action définie dessus}
Soit $\mathcal{X}$ une $k$-variété, $G$ un groupe algébrique réductif agissant sur $\mathcal{X}$. Notons $X$ comme dans la notation \ref{Notation du lieu stable}. 

Alors si l'on note $\mathcal{B} := \lbrace x\in X^\urcorner,$ l'action de $G_{\mathcal{H}(x)}^{an}$ n'est pas bien définie $\rbrace$, on peut définir une relation d'équivalence $\mathcal{G}$ sur $X^\urcorner \backslash \mathcal{B}$ où $x\mathcal{R}y \iff y \in G^{an}_{\mathcal{H}(x)}$. On notera $(X^\urcorner \backslash \mathcal{B})/\mathcal{G}$ l'espace quotient muni de la topologie quotient.
\end{defi}

%
%

\begin{defi}\label{defi morphisme entre compactification du quotient et quotient compactification}
Soit $\mathcal{X}$ une $k$-variété, $G$ un groupe algébrique réductif agissant sur $\mathcal{X}$. Notons $X$ comme dans la notation \ref{Notation du lieu stable}. 

Notons $\pi : X \rightarrow X/G$ le morphisme surjectif de type fini défini dans GIT. Alors l'analytification $\pi^{hyb} : X^{hyb} \rightarrow (X/G)^{hyb}$ est aussi surjective. Notons $\mathcal{F} := \pi^{hyb, -1}((X/G)^\beth)$, c'est un fermé de $X^{hyb}$ qui est invariant par le flot et qui contient $X^\beth$. Ainsi, l'on dispose d'un morphisme surjectif $X^{hyb} \backslash \mathcal{F} \rightarrow (X/G)^+$ qui est compatible avec le flot. 

Notons alors $F \subset X^\urcorner$ l'image de $X^+ \cap \mathcal{F}$ via le quotient par le flot.

Alors $\pi$ induit une application continue surjective :
\begin{center}
$\Pi : X^\urcorner \backslash F \rightarrow (X/G)^\urcorner$.
\end{center}
\end{defi}

\begin{prop}\label{homéo entre compactification du quotient et quotient compactification}
Soit $\mathcal{X}$ une $k$-variété, $G$ un groupe algébrique réductif agissant sur $\mathcal{X}$ et notons $X$ comme dans la notation \ref{Notation du lieu stable}. 

On reprends les notations des définitions \ref{définition lieu ou action est bien def et action définie dessus} et \ref{defi morphisme entre compactification du quotient et quotient compactification}.

Alors, le fermé $F$ de $X^\urcorner$ contient $\mathcal{B}$.

De plus, l'application $\Pi$ est invariante par la relation d'équivalence $\mathcal{G}$ et l'application :
\begin{center}
$\varpi : (X^\urcorner \backslash F)/\mathcal{G} \rightarrow (X/G)^\urcorner$
\end{center}
est une bijection continue qui se restreint en l'identité sur $(X/G)^{an}$ et qui est un homéomorphisme de $(\delta X\backslash F)/\mathcal{G}$ vers $\delta (X/G)$. En particulier $(\delta X\backslash F)/\mathcal{G}$ est compact. 
\end{prop}

%

\begin{proof}
Par définition des lieux stables et pré-stables et comme le morphisme $\pi : X \rightarrow X/G$ est affine, on peut se restreindre au cas où $X$ est affine et l'action de $G$ est fermée sur $X$.

Comme l'application $\pi^{hyb} : X^{hyb} \rightarrow (X/G)^{hyb}$ est continue et surjective, par construction de $\Pi$, cette dernière reste continue et surjective.

Il faut maintenant vérifier que $\mathcal{B} \subset F$. Soit $x\in \mathcal{B}$. Alors il existe $g\in G^{an}_{\mathcal{H}(x)}$ tel que $g\cdot x \in X^\beth$. Comme $\pi^{hyb}(x) = \pi^{hyb}(g\cdot x) \in \pi^{hyb}(X^\beth) \subset (X/G)^\beth$ on a bien $x \in F$.

Par Maculan (\cite{MaculanGIT}, voir le deuxième point du théorème \ref{GIT analytique}), l'application $\Pi$ est invariante par la relation d'équivalence $\mathcal{G}$ définie à la définition \ref{définition lieu ou action est bien def et action définie dessus} et la factorisation par $(X^\urcorner\backslash F)/\mathcal{G}$ est bijective et continue. 

De plus, par Maculan (\cite{MaculanGIT}), l'application $\pi^{an}_0 : X^{an}_0 \rightarrow (X/G)^{an}_0$ vérifie que si $U \subset X^{an}_0$ est $G$-invariante alors $\pi^{an}_0(U)$ est ouvert. Donc, $\varpi$ se restreint en un homéomorphisme de $(\delta X\backslash F)/\mathcal{G}$ vers $\delta (X/G)$. Pour conclure sur la compacité de $(\delta X\backslash F)/\mathcal{G}$ dans le cas où $X$ n'est pas affine, on utilise le fait que $X$ est une $k$-variété donc quasi-compact et donc il existe un nombre fini d'ouverts affines $U_i$, $G$-invariants, où l'action est fermée sur $U_i$ avec $X = \bigcup U_i$. Comme chacun des $(\delta U_i \backslash (F\cap U_i))\mathcal{G}$ est compact, $(\delta X\backslash F)\mathcal{G}$ est également compact. 
\end{proof}

Dans le cas où $\mathcal{B}$ et $F$ coïncident, on dispose même d'un résultat plus fort : la bijection continue est en fait un homéomorphisme.

\begin{prop}\label{lieu où l'action est bien définie est compact}
Soit $\mathcal{X}$ une $k$-variété, $G$ un groupe algébrique réductif agissant sur $\mathcal{X}$ et notons $X$ comme dans la notation \ref{Notation du lieu stable}. 

Supposons que $\mathcal{B} = F$. Alors, $(X^\urcorner \backslash \mathcal{B})/\mathcal{G}$ est compact et contient $(X/G)^{an}$ comme ouvert dense. En particulier, l'application $\varpi$ définie à la proposition \ref{homéo entre compactification du quotient et quotient compactification} est un homéomorphisme.
\end{prop}

\begin{proof}
Par la proposition \ref{homéo entre compactification du quotient et quotient compactification}, on sait déjà que que $(\delta X \backslash \mathcal{B})/\mathcal{G}$ est compact. 

Pour montrer la compacité de $(X^\urcorner \backslash \mathcal{B})/\mathcal{G}$, il suffit donc de prendre une suite $x_n \in (X/G)^{an} \subset(X^\urcorner \backslash \mathcal{B})/\mathcal{G}$ et de vérifier qu'elle possède une valeur d'adhérence. Si la suite $x_n$ possède une valeur d'adhérence dans $(X/G)^{an}$, c'est terminé. Sinon, on peut prendre $y_n \in X^{an}$ des relevés de $x_n$. La suite $y_n \in X^{an} \subset X^\urcorner$ a alors une valeur d'adhérence $y$ dans $X^\urcorner$. Par le théorème \ref{equivalence dégénérecence quotient et action bien déf}, $y \notin \mathcal{B}$ donc son image dans $(X^\urcorner \backslash \mathcal{B})/\mathcal{G}$ est une valeur d'adhérence des $x_n$ dans $(X^\urcorner \backslash \mathcal{B})/\mathcal{G}$.
\end{proof}

On va donc maintenant montrer que $\mathcal{B}$ et $F$ coïncident.

\begin{prop}\label{on peut caractériser le fermé qu on doit retirer dans la compactification}
Soit $\mathcal{X}$ une $k$-variété intègre, $G$ un groupe algébrique réductif agissant sur $\mathcal{X}$ et notons $X$ comme dans la notation \ref{Notation du lieu stable}, alors $F = \mathcal{B}$. En particulier, $\mathcal{B}$ est fermé.

En particulier, si $X$ est un $k$-schéma affine, intègre de type fini sur $k$ et $G$ un groupe algébrique agissant sur $X$ dont l'action est fermée, alors $F = \mathcal{B}$. Donc, $\mathcal{B}$ est fermé.
\end{prop}

\begin{proof}
Par définition des lieux stables et préstables à la définition \ref{Définition lieu stable}, on peut se ramener au cas où $X$ est un schéma affine, intègre de type fini sur $k$ et $G$ agit sur $X$ avec une action fermée.

On sait déjà que $\mathcal{B} \subset F$. Soit donc $x\in F$ alors $\pi^{hyb}(x) \in (X/G)^\beth$. Par la proposition \ref{si le morphisme est surjectif, surjectivité des beth}, appliquée au morphisme surjectif $\pi : X \rightarrow X/G$, il existe $y\in X^\beth$ tel que $\pi^{hyb}(y) = \pi^{hyb}(x)$. Donc, par Maculan (\cite{MaculanGIT}), $\exists g\in G^{an}_{\mathcal{H}(x)}$ tel que $y = g\cdot x$ et donc l'action n'est pas bien définie sur en $x$.
\end{proof}


On dispose donc de deux compactifications homéomorphes de $(X/G)^{an}$.

\begin{thm}\label{Si le morphisme est equidimensionnel compactifications sont homeo}
Soit $\mathcal{X}$ une $k$-variété, intègre, $G$ un groupe algébrique réductif agissant sur $\mathcal{X}$ et notons $X$ le lieu stable vis à vis d'un faisceau inversible sur $\mathcal{X}$ comme dans la notation \ref{Notation du lieu stable}. 

Alors $\varpi$ de la proposition \ref{homéo entre compactification du quotient et quotient compactification} est un homéomorphisme :
\begin{center}
$(X^\urcorner \backslash \lbrace x\in X^\urcorner,$ l'action de $G^{an}_{\mathcal{H}(x)}$ n'est pas bien définie$\rbrace)/\mathcal{G} \rightarrow (X/G)^\urcorner$
\end{center}
qui se restreint en l'identité sur $(X/G)^{an}$. Les deux compactifications de $(X/G)^{an}$ sont donc homéomorphes.
\end{thm} 

\begin{proof}
En utilisant la proposition \ref{on peut caractériser le fermé qu on doit retirer dans la compactification}, la proposition \ref{homéo entre compactification du quotient et quotient compactification} dit que $\varpi$ est une bijection continue de $(X^\urcorner \backslash \lbrace x\in X^\urcorner,$ l'action de $G^{an}_{\mathcal{H}(x)}$ n'est pas bien définie$\rbrace)/\mathcal{G}$ vers $(X/G)^\urcorner$. 

Or par la proposition \ref{lieu où l'action est bien définie est compact}, $(X^\urcorner \backslash \lbrace x\in X^\urcorner,$ l'action de $G^{an}_{\mathcal{H}(x)}$ n'est pas bien définie$\rbrace)/\mathcal{G}$ est compact, donc c'est en fait un homéomorphisme.
\end{proof}

\begin{cor}
Soit $X$ un $k$-schéma affine, intègre de type fini, $G$ un groupe algébrique réductif agissant sur $X$ dont l'action est fermée. Alors $\varpi$ de la proposition \ref{homéo entre compactification du quotient et quotient compactification} est un homéomorphisme :
\begin{center}
$(X^\urcorner \backslash \lbrace x\in X^\urcorner,$ l'action de $G^{an}_{\mathcal{H}(x)}$ n'est pas bien définie$\rbrace)/\mathcal{G} \rightarrow (X/G)^\urcorner$
\end{center}
qui se restreint en l'identité sur $(X/G)^{an}$. Les deux compactifications de $(X/G)^{an}$ sont donc homéomorphes.
\end{cor} 

On peut appliquer ce théorème pour avoir une version plus faible du résultat de Maculan (voir Théorème \ref{GIT analytique}) sur l'ouverture du morphisme $\pi^{an} : X^{an} \rightarrow (X//G)^{an}$ dans le cadre hybride.

\begin{prop}
Soit $X$ un schéma affine sur un corps non-trivialement valué $k,~ G$ un groupe réductif algébrique agissant sur $X$ dont l'action est fermée. Soit $V \subset X^{hyb}$ un ouvert flottant, $G$-invariant de $X^{hyb} \backslash (\pi^{hyb})^{-1}((X/G)^\beth)$, alors $\pi^{hyb}(V)$ est ouvert dans $(X/G)^{hyb}$.
\end{prop}

\begin{proof}
Comme $V \cap (\pi^{hyb})^{-1}((X/G)^\beth) = \emptyset$ et $V$ est ouvert, $\Phi(V)$ est ouvert dans $X^\urcorner\backslash \mathcal{B}$. 

Ainsi, $\rho(\Phi(V))$ est ouvert dans $(X^\urcorner \backslash \mathcal{B})/\mathcal{G}$, puisqu'il est $G$-invariant et $(X^\urcorner \backslash \mathcal{B})/\mathcal{G}$ est équipé de la topologie quotient induite par $\rho : X^\urcorner \backslash \mathcal{B}\rightarrow (X^\urcorner \backslash \mathcal{B})/\mathcal{G}$. 

Alors, $\omega \circ \rho(\Phi(V))$ est ouvert dans $(X/G)^\urcorner$ et $\pi^{hyb}(V) = \Phi^{-1}(\omega \circ \rho(\Phi(v)))$ est ouvert dans $(X/G)^{hyb} \backslash (X/G)^\beth$. Puisque $(X/G)^\beth$ est fermé dans $(X/G)^{hyb}, \pi^{hyb}(V)$ est ouvert dans $(X/G)^{hyb}$.
\end{proof}

\section{Application aux fractions rationnelles}\label{ApplicationAuxFractionsRationnelles}

Dans cette section, on va s'intéresser à l'espace des fractions rationnelles de degré $d \geq 1$ que l'on note $Rat_d$. C'est un ouvert de l'espace projectif de dimension $2d-1$. J. Silverman (\cite{SilvermanSpaceRationalMaps}) a montré que l'on pouvait le voir comme un schéma défini sur $\Z$ avec 
\begin{center}
$\Rat_d := \Spec~\Z[\frac{a_0^{i_0}\cdots a_d^{i_d}b_0^{j_0}\cdots b_d^{j_d}}{\rho}]_{i_0 + \cdots + i_d + j_0 + j_d = 2d}$
\end{center}  
où $a_l, b_k$ sont les coefficients de $f = \frac{P(z)}{Q(z)} = \frac{a_0 + a_1z + \cdots + a_dz^d}{b_0 + b_1z + \cdots + b_dz^d}$ et $\rho$ est le résultant de $P$ et $Q$. Dans la suite, on prendra $\Rat_d$ comme étant un schéma défini sur $k$ un corps muni d'une valeur absolue non-triviale et pour simplifier les notations, on notera $\Rat_d = Spec~k[\frac{\underline{a}\underline{b}}{\rho}]$.

On dispose d'une action de $GL_2$ sur $\Rat_d$ : l'action par conjugaison. Cette action se lit sur les coefficients des fractions rationnelles de la façon suivante : 

\begin{lem}\label{invariance du resultant si M est dans SL2}
Soient $ \Phi = \frac{a_0 z^d + \cdots + a_d}{b_0 z^d + \cdots + b_d}, M =  
\begin{pmatrix}
\alpha & \beta\\
\gamma & \delta\\
\end{pmatrix}$, alors :

\begin{equation*}
 \Phi^M = \frac{(\delta \sum_0^d a_k \alpha^{d-k}\gamma^k - \beta \sum_0^d b_k \alpha^{d-k}\gamma^k)z^d + \cdots + (\delta \sum_0^d a_k \beta^{d-k}\delta^k - \beta \sum_0^d b_k \beta^{d-k}\delta^k)}{(\alpha \sum_0^d a_k \alpha^{d-k}\gamma^k - \gamma \sum_0^d b_k \alpha^{d-k}\gamma^k)z^d + \cdots + (\alpha \sum_0^d a_k \beta^{d-k}\delta^k - \gamma \sum_0^d b_k \beta^{d-k}\delta^k)}.
\end{equation*} 

Une fois le polynôme mis sous cette forme, on a $\mathrm{res}(\Phi^M) = \mathrm{res} (\Phi) \det(M)^{d^2+d}$ (voir l'exercice 2.7 de \cite{SilvermanLivre}). Donc si on prend $M \in SL_2$, alors on a invariance du résultant.
\end{lem}

Dans toute la suite, on s'intéressera donc à l'action de $SL_2$ sur $\Rat_d$, on notera $M_d$ le quotient de $\Rat_d$ par $SL_2$, c'est un schéma de type fini sur $k$. 

\begin{lem}
L'action de $SL_2$ sur $\Rat_d$ est propre. En particulier, l'action est fermée. 
\end{lem}
\begin{proof}
Ce résultat est dû à la proposition 0.8 de \cite{GIT}. On pourra se référer au lemme 2.4 de Favre-Gong dans \cite{FavreGong} pour une application de cette proposition au cas particulier des fractions rationnelles.
\end{proof}

On va s'intéresser à la partie $\Rat_d^\beth$ qui est la partie que l'on retire de $\Rat_d^{hyb}$.

Rappelons tout d'abord une définition classique dans le cadre où l'on étudie des fractions rationnelles sur un corps non-archimédien.

\begin{defi}\label{def bonne réduc}
Soit $k$ un corps valué non-archimédien et soit $f = \frac{a_0 + a_1z + \cdots + a_dz^d}{b_0 + b_1z + \cdots + b_dz^d} \in \Rat_d(k)$. On peut supposer que $max~(|a_i|, |b_i|) = 1$, on peut alors considérer $\tilde{f} \in \Rat_d(\tilde{k})$ où l'on réduit les coefficients de $f$ dans $\tilde{k}$ le corps résiduel de $k$.

Alors on dit que $f$ a bonne réduction si et seulement si $\tilde{f}$ a degré exactement $d$. 
\end{defi}

\begin{prop}\label{caractérisation du beth de rat}
Soit $(k, |\cdot|)$ un corps non-trivialement valué.

La partie $\Rat_d^\beth$ correspond à l'ensemble des fractions rationnelles $f$ définies sur une extension valuée de $(k, |\cdot|_0)$ ayant bonne réduction.

Ainsi, l'action de $\SL^{an}_{2,\mathcal{H}(f)}$ n'est pas bien définie sur $f$ si et seulement si $f$ a potentielle bonne réduction. 
\end{prop}

\begin{proof}
L'espace $\Rat_d^\beth$ est un fermé de $\Rat_d^{an, |\cdot|_0}$, où l'analytification est faite selon la valeur absolue triviale. Il correspond aux fractions rationnelles dont la valeur absolue des coefficients est plus petite que 1. Ainsi, 
\begin{equation*}
f \in \Rat_d^\beth \iff \forall \frac{\underline{a} \underline{b}}{\rho}, \vert \frac{\underline{a} \underline{b}}{\rho} \vert \leq 1.
\end{equation*}

Comme $\rho$ est une combinaison linéaire des $\underline{a}\underline{b}$, il existe une combinaison linéaire des $\frac{\underline{a}\underline{b}}{\rho}$ tel que $1 = \sum \lambda_{I,J} \frac{\underline{a}\underline{b}}{\rho}$ et donc nécessairement l'un des $\frac{\underline{a}\underline{b}}{\rho}$ a valeur absolue 1. Donc, $\Rat_d^\beth = \lbrace |\cdot|, \max|\frac{\underline{a}\underline{b}}{\rho}| = 1\rbrace$ et par multiplicativité des valeurs absolues $\Rat_d^\beth = \lbrace |\cdot|, \max|\frac{a_i^{2d}}{\rho}|, |\frac{b_i^{2d}}{\rho}| = 1\rbrace$. Donc, si $f \in \Rat_d^\beth$ et que l'on prend ses coefficients de façon à ce que le maximum des coefficients soit de norme 1, alors $|\rho| = 1$ et donc $f$ reste de degré $d$ dans le corps résiduel. 
\end{proof}

On peut donc maintenant combiner le corollaire \ref{equivalence dégénérecence quotient et action bien déf} et la proposition \ref{caractérisation du beth de rat} pour retrouver les résultats de Favre-Gong dans le contexte de la compactification hybride. 

\begin{prop}\label{pot bonne reduc ssi degenere pas quotient}
Soit $f_n \in \Rat_d^{an}$ où l'analytification est prise au sens de la valeur absolue usuelle sur $k$ telles que $f_n \rightarrow f \in \Rat_d^\urcorner$. Notons $\pi^{an} : \Rat_d^{an} \rightarrow \M_d^{an}$ la projection, alors
\begin{center}
L'action de $\SL^{an}_{2,\mathcal{H}(f)}$ est bien définie $\iff f$ n'a pas potentielle bonne réduction$ \iff \pi^{an}(f_n) \rightarrow \infty$.
\end{center} 
\end{prop}

\begin{rem}
Si $f =\frac{a_0 + a_1z+ \cdots + a_dz^d}{b_0 + b_1z + \cdots + b_dz^d}$, alors on définit $|\rho_f| =  \min (|\frac{\rho}{a_i^{2d}}|, |\frac{\rho}{b_i^{2d}}|)$ et on dit que $f$ est de résultant maximal dans sa fibre si $|\rho_f| = \max\lbrace |\rho_g|, \exists M\in \SL_2, g^M = f \rbrace$.
 
L'une des façons d'assurer que $\pi^{an}(f_n) \rightarrow \infty$ est de prendre $f_n$ de résultant maximal dans sa fibre. En effet, comme $f_n \rightarrow \infty \iff |\rho_{f_n}|\rightarrow 0$, alors si $f_n$ est de résultant maximal dans sa fibre, alors pour tout $M_n \in \SL_2, |\rho_{f_n^{M_n}}| \leq |\rho_{f_n}| \rightarrow 0$, donc aucune suite $f_n^{M_n}$ n'a de valeurs d'adhérence dans $\M_d^{an}$.  
\end{rem}

\begin{rem}
Si l'on prends une suite de fractions rationnelles $f_n$ telle que la suite $\pi^{an}(f_n)$ dégénère, alors on peut prendre un ultra-filtre $\omega$ non-principal et regarder la limite $f_\omega \in \Rat_d^\urcorner$ des $f_n$ le long de l'ultra-filtre $\omega$. Une différence entre cette limite et celle obtenue par Favre-Gong \cite{FavreGong} est que son corps résiduel est un corps plus petit que dans leurs travaux. Ici, le corps résiduel sur lequel $f$ est défini est la complétion d'un corps de degré de transcendance au plus $2d-1$ sur $\C$. En particulier, le groupe de valeur de la clôture algébrique de son corps résiduel est l'ensemble des nombres positifs d'un $\Q$-espace vectoriel de dimension au plus $2d-1$. Alors que le corps sur lequel est défini les fractions rationnelles de Favre-Gong est un corps de Robinson et a pour groupe des valeurs tout $\R_+$. Le fait d'avoir un corps plus petit et en particulier de degré de transcendance fini peut-être très utile comme montré par C. Gong \cite{GongMultiplierScale}.
\end{rem}

La proposition \ref{propreté morphisme sur le quotient} permet de retrouver un résultat de L. DeMarco (proposition 4.1 de \cite{DeMarcoSpaceQuadraticRationalMaps}) qui n'avait été prouvé que pour $k = \C$ bien que Favre-Gong (\cite{FavreGong}) aient indiqué que le résultat devrait être vrai pour tout corps $k$.

\begin{cor}\label{propreté itération} Proposition 4.1 de DeMarco \cite{DeMarcoSpaceQuadraticRationalMaps}.

Soit $l \in \N^*, d\geq 2$, l'application itération $I_l : M_d^{an} \rightarrow M_{d^l}^{an}$ est propre.
\end{cor}

\begin{proof}
Considérons la fonction itération $g_l : \Rat_d \rightarrow \Rat_{d^l}$, alors cette fonction vérifie que $g_l (f^M) = g_l(f)^M$ pour tout $f\in \Rat_d, M\in \SL_2$ et par la proposition \ref{propreté morphisme sur le quotient} il suffit de montrer que $g_l^{-1}(\Rat_{d^l}^\beth) \subset \Rat_d^\beth$. Par la proposition \ref{caractérisation du beth de rat}, cela revient à dire que si une fraction rationnelle $f$ est telle que $f^l$ a bonne réduction, alors c'est le cas de $f$. C'est un résultat dû à R. Benedetto (Corollary 8.14, \cite{BenedettoLivre}).
\end{proof}

\begin{rem}
Favre-Gong avaient déjà remarqué que cet argument permettait de montrer la propreté de l'application itération.
\end{rem}

La propreté de cette application permet d'en déduire le corollaire suivant.

\begin{cor}\label{existence itération compactification}
Soit $l \in \N^*$, alors l'application $I_l : \M_d \rightarrow \M_{d^l}$ s'étend à $\M_d^\urcorner \rightarrow \M_{d^l}^\urcorner$. 
\end{cor}

Donc, la compactification $\M_d^\urcorner$ vérifie l'une des conditions demandées par DeMarco \cite{DeMarcoSpaceQuadraticRationalMaps}. De plus, même si dans le cas $d=2$, l'on ne possède pas de projection vers l'espace des mesures de probabilité barycentrée $\overline{M^2_{bc}(\p^1_\C)}/SO_3$ chacune des fractions rationnelles est bien associée à une unique mesure de probabilité sur un espace de Berkovich. Poineau a montré que ces mesures de probabilités bougeaient continûment sur $\Rat_d^\urcorner$ \cite{PoineauDynamique}.

Comme l'action de $\SL_2$ sur $\Rat_d$ est fermée, on peut caractériser le fermé $F$ à enlever de $\Rat_d^\urcorner$ de manière à avoir une bijection continue de $\Rat_d^\urcorner \backslash F \rightarrow \M_d^\urcorner$. 


En combinant les propositions \ref{on peut caractériser le fermé qu on doit retirer dans la compactification} et \ref{caractérisation du beth de rat}, on obtient le résultat suivant :
 
\begin{prop}\label{bijection entre Rat_d quot et M_d}
L'ensemble $\lbrace f\in \delta \Rat_d,$ f a potentielle bonne réduction$\rbrace$ est un fermé de $\Rat_d^\urcorner$.

De plus, on dispose d'un homéomorphisme :
\begin{center}
$(\Rat_d^\urcorner \backslash \lbrace f\in \delta \Rat_d,$ f a potentielle bonne réduction$\rbrace)/\SL_2 \rightarrow \M_d^\urcorner$
\end{center}
qui est l'identité sur $\M_d^{an}$.
\end{prop}
\begin{proof}
En effet, l'on dispose d'un homéomorphisme
\begin{center}
$(\Rat_d^\urcorner \backslash \lbrace f\in \delta \Rat_d,$ l'action de $SL_{2,\mathcal{H}(f)}^{an}$ n'est pas bien définie$\rbrace)/\SL_2 \rightarrow \M_d^\urcorner$, 
\end{center}
par le théorème \ref{Si le morphisme est equidimensionnel compactifications sont homeo}. Enfin, par la deuxième partie de la proposition \ref{caractérisation du beth de rat}, on voit que $f$ a potentielle bonne réduction si et seulement si l'action de $\SL_{2,\mathcal{H}(f)}^{an}$ n'est pas bien définie.
\end{proof}

\bibliographystyle{alpha}
\bibliography{biblio_these.bib}
\end{document}